\documentclass[reqno]{amsart}
\usepackage[foot]{amsaddr}
\usepackage[a4paper,hmargin=2cm,vmargin=3cm]{geometry}
\usepackage{enumitem}
\usepackage[T1]{fontenc}
\usepackage[utf8]{inputenc}
\usepackage[english]{babel}
\usepackage{amsmath}
\usepackage{amsthm}
\usepackage{amssymb}
\usepackage{setspace}
\usepackage{bm}
\usepackage{color}

\DeclareMathOperator{\sign}{sign}

\usepackage[pdfusetitle,
 bookmarks=true,bookmarksnumbered=false,bookmarksopen=false,
 breaklinks=false,pdfborder={0 0 1},backref=false,colorlinks=false]
 {hyperref}

\makeatletter
\numberwithin{equation}{section}
\numberwithin{figure}{section}

\theoremstyle{plain}
\newtheorem{thm}{\protect\theoremname}[section]
\newtheorem{lem}[thm]{\protect\lemmaname}
\newtheorem{proposition}[thm]{Proposition}
\newtheorem{cor}[thm]{Corollary}

\newtheorem{assu}[thm]{Assumption}
\theoremstyle{remark}
\newtheorem*{remark}{Remark}

\newcommand{\E}{\mathbb{E}}
\makeatother

\providecommand{\lemmaname}{Lemma}
\providecommand{\theoremname}{Theorem}
\providecommand{\definitionname}{Definition}
\newcommand{\N}{\mathbb{N}}
\newcommand{\R}{\mathbb{R}}
\renewcommand{\(}{\left(}
\renewcommand{\)}{\right)}
\renewcommand{\[}{\left[}
\renewcommand{\]}{\right]}

\newcommand{\Ph}{\mathbf{P}_h}
\newcommand{\Ps}{\mathbf{P}_\sigma}

\newcommand{\Qs}{\mathbf{Q}^{(n\rho)}_{\sigma}}
\newcommand{\D}{\mathbf{D}}
\newcommand{\Gs}{\mathbf{G}_{n}}
\newcommand{\Hn}{\mathbf{H}_n}
\newcommand{\PPP}{\mathbf{PPP}}
\newcommand{\e}{\varepsilon}
\newcommand{\gas}{\Ph-\text{a.s.}}
\newcommand{\Gas}{\Ph-\text{almost surely}}
\newcommand{\geas}{\Ph-\text{e.a.s.}}
\newcommand{\Geas}{\Ph-\text{eventually almost surely}}
\newcommand{\Gae}{\Ph-\text{almost every}}
\newcommand{\res}{\tfrac{1}{n}}
\newcommand{\MGF}{\text{$\log-$MGF}}
\newcommand{\FLT}{\text{FLT}}
\newcommand{\1}[1]{\bm{1}_{\{#1\}}}

\def\been#1{\begin{equation}#1\end{equation}}
\def\aled#1{\begin{aligned}#1\end{aligned}}
\usepackage{lineno}

\title{REM universality for linear random energy}

\author{Francesco Concetti}
\address{Faculty of Mathematics and Computer Science, UniDistance Suisse, 3900 Brig, Switzerland}
\email{francesco.concetti@unidistance.ch}

\author{Simone Franchini}
\address{CNR-ISTC, Via Gian Domenico Romagnosi 18, 00196 Rome, Italy}

\begin{document}

\begin{abstract}
We consider a sequence of random Hamiltonians $H_n(h,\sigma)=\sum^n_{i=1}h_i(\sigma_i-m)$, and study the asymptotic ($n\to \infty$) distribution of the energy levels $\(H_n(h,\sigma)\)_{\sigma\in \{-1,1\}^n}$, where $h_1,h_2,\cdots$ are i.i.d. random variables. We show that, when $e^{O(n)}$ configurations are sampled at random, the corresponding collection of energy levels converges in distribution to a Poisson point process with exponential intensity measure. This establishes the Random Energy Model (REM) universality for the present model. Our results strengthen earlier works on local REM universality by characterizing the distribution of $O(1)-$order fluctuations of $H_n$. In addition, we improve upon the REM universality by dilution studied by Ben Arous, Gayrard, Kuptsov by allowing an exponentially large number $e^{O(n)}$ of sampled configurations, instead of $e^{o(\sqrt{n})}$. Finally, we derive the asymptotic distribution of the Gibbs weight.
\end{abstract}
\maketitle
\section{Introduction}
Let $h := (h_i)_{i\in\N}$ be a sequence of independent and identically distributed (i.i.d.) real-valued random variables, and let
$\sigma := (\sigma_i)_{i\in\N}$ be a sequence of i.i.d. $\{-1,1\}-$valued random variables,
independent of $h$. 
We denote by $\Ps$ the distribution of $\sigma$, characterized by
\been{
\Ps(\sigma_1=1)=\frac{1+m}{2},
}
for some $m\in(-1,1)$. We denote by $\Ph$ the joint distribution of $h$ and denote by $\E\[\,\cdot\,\]$ the expectation over $\Ph$. For any $n\in\mathbb N$, we define
\been{
H_n(h,\sigma) := \sum_{i=1}^n h_i (\sigma_i - m), \qquad \sigma \in \{-1,1\}^{\N}.
}
The Hamiltonian $H_n$ provides a simple example of a random Hamiltonian whose
energy levels $\bigl(H_n(h,\sigma)\bigr)_{\sigma\in\{-1,1\}^{\N}}$ are correlated random variables.
Models of this type arise naturally in the statistical mechanics of disordered systems,
notably in spin glasses, as well as in combinatorial optimization problems.
In particular, $H_n$ is closely related to the number partitioning problem
\cite{Mertens0}.

It has long been conjectured that, for a broad class of random Hamiltonians, the properly rescaled energy levels converge in distribution to a Poisson point process (PPP).~\cite{Mertens0,Mertens00,Mertens1}.
Consequently, the asymptotic statistics of the energy levels coincide with those of
Derrida's Random Energy Model (REM): a spin-glass model in which the energy levels are
independent by construction \cite{Derrida}. This conjecture is commonly referred to as \emph{REM universality}.

In the original REM, the energies are Gaussian random variables. It was later shown that the convergence of rescaled energy levels to a PPP holds for a broader class of models with independent energies drawn from more general distributions~\cite{mez_rem}.

In a series of works, Borgs, Chayes, Mertens, and Nair for the partitioning~\cite{Borgs0,Borgs1}, and Bovier and Kourkova for more general spin glass Hamiltonians~\cite{Bovier0}, proved that the fluctuations of the energy levels converges to a PPP, when observed in a small window of the spectrum whose width
shrinks exponentially fast with the system size $n$. They called this local property of the energy spectrum \emph{local REM universality}.

A complementary perspective was later introduced by Ben~Arous, Gayrard, and
Kuptsov \cite{ArousREM}, who established \emph{REM universality by dilution}. Specifically, they proved that REM universality persists for energy levels arising from random subsets of configurations whose cardinality is sub-exponential ($\sim e^{o(\sqrt{n})}$). 

The present work substantially extends these results. For the Hamiltonian $H_n$, we establish REM universality for energy
fluctuations of order $1$ and for families of configurations whose cardinality grows exponentially with the system size ($\sim e^{\alpha n}$).
In particular, this proves REM universality for an extensive portion of the
energy levels.

Recently, REM universality has attracted renewed interest in physics literature, particularly in connection with advances in mean-field spin glass theory, where new methods that exploit REM-like behavior are proposed \cite{Franchini2021,RSBwr,FranchiniSPA2023}. 

The results of this manuscript are based on the following assumption.
\begin{assu}
\label{assum:1}
The distribution of $h_1$ has an absolutely continuous part, there exists $\e>0$ and a constant $p_1>0$ such that  $\Ph(|h_1|<t)\leq p_1t$ for all $t\in [0,\e)$, and there exists an interval $[c,d]$ such that the density of  $h_1$ on $[c,d]$ is bounded from below by a constant $p_2>0$. Moreover, the first, second, and third moments exist with
\been{
\mathbb E[h_1] = \psi_1, \quad
\mathbb E[h_1^2] = \psi_2, \quad
\mathbb E[|h_1|] = \psi_3,\quad \mathbb E[|h_1|^3] = \psi_4.
}
\end{assu}
Given $\lambda>0$, we define the locally finite measure $\D_{\lambda}$ on $\R$ by
\been{
\D_{\lambda}(\mathfrak{U}):=\int_\mathfrak{U} e^{-\lambda x}dx,
}
for any Borel set $\mathfrak{U}$ in the Borel $\sigma-$algebra $\mathcal{B}(\R)$. The first result of the manuscript is the following. Define the functions 
\been{
\label{eq:defGG20}
G(\lambda):=\E\[\log((1+m) e^{\lambda (1-m)h_1}+(1-m)e^{-\lambda (1+m)h_1})\]-\log(2),\qquad G^*(a):=\sup_{\lambda\in \R}\(\lambda a-G(\lambda)\),
}
and denote by $G'$ the first derivative of $G$. We also define the quantities
\been{
\varsigma:=-m\psi_1+\psi_3,\qquad \gamma:=-\E[\log(1+\sign(h_1)m)]+\log(2),\quad \Gamma_{n}(h):=-\sum^n_{i=1}\log\(1+\sign(h_i)m\)+n\log(2)
}
Our first result establishes a universal asymptotic behavior of the distribution of $H_n$, conditionally on $h$.
\begin{thm}
\label{them:0}
Assume that Assumption \ref{assum:1} holds. Given a deterministic number $c\in (0,\gamma)$, there exists a unique $\tilde a\in (0,\varsigma)$ and $\tilde \lambda\in \R_{>0}$ (depending on $c$) such that
\been{
\label{eq:solac0}
G^*(\tilde a)=c,\quad G'(\tilde \lambda)=\tilde a.
}
Moreover, for any random sequence $(C_n)_{n\in \N}$ such that
\been{
\label{eq:CcondREM0}
C_n(h) \in ( 0, \Gamma_n(h)),\quad \geas,\qquad\textup{and}\qquad \lim_{n\to\infty}\tfrac{1}{n}C_n=c,\quad \gas,
}
there exists a $h-$measurable random sequence $(A_n)_{n\in \N}$ such that the sequence of measure kernels $\(\mathbf{K}_n \)_{n\in \N}$, defined by
\been{
\label{eq:Kconve}
\mathbf{K}_n(h,\mathfrak{U}):=e^{C_n(h)}\Ps(\{\sigma:\,H_n(h,\sigma)-A_n(h)\in\mathfrak{U}\}),\quad h\in \R^{\N},\,\mathfrak{U}\in \mathcal{B}(\R),
}
converges vaguely $\Gas$ to the deterministic measure $\D_{\tilde \lambda}$.
\end{thm}
\begin{remark}
The main novelty of Theorem \ref{them:0}, with respect to \cite{ArousREM}, is that the centering sequence $(A_n)_{n\in\N}$ is allowed to depend on the environment $h$. This random centering is needed to obtain the convergence of the kernels $\mathbf K_n(h,\cdot)$.
\end{remark}
Given $n\in\mathbb N$ and $\sigma\in\{-1,1\}^{\mathbb N}$, let $\sigma_{[n]}$ denote its projection onto $\{-1,1\}^n$, namely
\been{
\sigma_{[n]} := (\sigma_1,\ldots,\sigma_n).
}
For $\tau\in\{-1,1\}^n$, we define the associated $n$-dimensional cylinder by
\been{
[\tau] := \{\sigma\in\{-1,1\}^{\mathbb N} : \sigma_{[n]}=\tau\}.
}
For fixed $h\in\mathbb R^{\mathbb N}$, the map $\sigma \mapsto H_n(h,\sigma)$ depends only on $\sigma_{[n]}$ and is therefore constant on each $n$-dimensional cylinder.

Fix $\rho\in(0,1)$. For each fixed $n\in\N$, let $\Qs$ denote the finite measure on the
$\sigma$-algebra on $\{-1,1\}^{\N}$ generated by the $n$-dimensional
cylinders, defined by
\begin{equation}
\Qs([\tau])
:= e^{n\rho(\log 2-\log(1+|m|))}\Ps([\tau]),
\qquad \tau\in\{-1,1\}^n.
\end{equation}

Throughout the paper, we use the symbol $\sigma$ to denote an entire infinite configuration in $\{-1,1\}^{\mathbb N}$ and work with measures defined on this space.

Let $\Omega_n \subset \{-1,1\}^{\mathbb N}$ be a set containing exactly one representative from each $n$-dimensional cylinder. Equivalently, $\Omega_n$ consists of configurations for which the coordinates $\sigma_{n+1},\sigma_{n+2},\dots$ are fixed, so that only the first $n$ spins vary. Whenever only the first $n$ coordinates are relevant--such as in the definition of the point process below--summation is taken over configurations in $\Omega_n$.

Let $(U_\sigma)_{\sigma\in\{-1,1\}^{\N}}$ be a family of independent random variables, uniformly distributed on $[0,1]$, and independent of both $h$ and $\sigma$.

We now state the main theorem of the paper, establishing REM universality for $H_n$. We say that a random variable is $h-$measurable if it is measurable with respect to the $\sigma-$algebra generated by $h$,
\begin{thm}[REM universality]
\label{them:1}
Given $\rho\in (0,1)$, and let $\tilde a$ and $\tilde \lambda$ satisfy
\been{
\label{eq:solac1}
G^*(\tilde a)=\rho\(\log(2)-\log(1+|m|)\,\),\qquad G'(\tilde \lambda)=\tilde a.
}
Then there exists a $h-$measurable random sequence $(A_n(h))$, such that the point process $\Hn$, defined by
\been{
\Hn(\mathfrak{U}):=\sum_{\sigma\in \Omega_n}\1{U_{\sigma}<\Qs([\sigma_{[n]}])}\1{H_n(h,\sigma)-A_n(h)\in \mathfrak{U}},\quad \mathfrak{U}\in \mathcal{B}(\R),
}
converges in distribution to a PPP with intensity measure $\D_{\tilde \lambda}$.
\end{thm}
\begin{remark}
The indicators $\1{U_{\sigma}<\Qs([\sigma_{[n]}])}$ implement a \emph{thinning} of the spin configurations, randomly reducing the number of configurations contributing to the point process $\Hn$. Unlike the REM universality by dilution of Ben~Arous, Gayrard, and
Kuptsov \cite{ArousREM}, which retains only $e^{o(\sqrt{n})}$ configurations, this thinning operation preserves, on average, $e^{n\rho(\log(2)-\log(1+|m|))}$ configurations.
\end{remark}
The above theorem has the following immediate corollary. Given a realization of the random sequence $U$, define the set of retained configurations $\mathcal{G}_n(U)\subseteq \Omega_n$ as
\been{
\mathcal{G}_n(U):=\{\sigma\in \Omega_n:\,U_{\sigma}\leq \Qs([\sigma_{[n]}])\}
}
If $\mathcal{G}_n(U)\neq \emptyset$, for $\beta>0$ and $\sigma \in \mathcal{G}_n(U)$, we define the \emph{Gibbs weight} of $\sigma$ as
\begin{equation}
\label{eq:gibbs-weights}
\Gs(\sigma)
:= \frac{e^{\beta H_n(h,\sigma)}}{\sum_{\tau \in\mathcal{G}_n(U)} e^{\beta H_n(h,\tau)}}.
\end{equation}
Hence, by reordering the sequence $(\Gs(\sigma))_{\sigma\in \mathcal{G}_n(U)}$ as a non-increasing sequence $(w_\alpha)_{\alpha\leq |\mathcal{G}_n(U)|}$, and set $w_\alpha=0$ for $\alpha>|\mathcal{G}_n(U)|$ (and $w_\alpha=0$ for any $\alpha$ if $\mathcal{G}_n(U)=\emptyset$)
\begin{cor}[Convergence to Poisson-Dirichlet]
If $\beta>\tilde{\lambda}$, the law of the sequence $(w_{\alpha})_{\alpha\in \N}$ converges to the Poisson-Dirichlet distribution $\textup{PD}(\tilde{\lambda}/\beta,0)$, where $\tilde{\lambda}$ is defined in \eqref{eq:solac1}.
\end{cor}
\begin{remark}
Note that the above theorem does not involve the sequence $(A_n)_{n\in \N}$. For a definition of the distribution $\textup{PD}(\tilde{\lambda}/\beta,0)$, see \cite[Equation $(3)$, Definition 1 and Corollary 9]{Pitman}.
\end{remark}
\begin{proof}
Let $(A_n)_{n\in \N}$ be the sequence defined in Theorem \ref{them:1}. We have the following equivalence
\been{
\Gs(\sigma)= \frac{e^{\beta (H_n(h,\sigma)-A_n(h))}}{\sum_{\tau \in \mathcal{G}_n(U)} e^{\beta (H_n(h,\tau)-A_n(h))}}.
}
 Thus, Theorem \ref{them:1} and \cite[Lemma 1.2.3]{Talcavi} complete the proof.
\end{proof}
{
\subsection{Sketch of the proof}

The proof of Theorem \ref{them:1} is based on the computation of the Laplace transform of the point process $\Hn$ (see formula \eqref{eq:RLaplace_trasf}).

For any measurable set $\mathfrak{U}$ and fixed $h\in\R^{\N}$, the random variable $\Hn(\mathfrak{U})$ is a weighted sum of $2^n$ independent Bernoulli random variables indexed by $\sigma\in\Omega_n$, each with parameter $\Qs([\sigma_{[n]}])$ and weight $\1{H_n(h,\sigma)-A_n(h)\in\mathfrak{U}}$. 
If $\rho\in(0,1)$, the parameters $\Qs([\sigma_{[n]}])$ decay exponentially fast in $n$ (see Lemma \ref{lem:Qbounds}). Hence, by Le Cam's Poisson approximation theorem \cite{LeCam1960}, conditionally on $h$, the variable $\Hn(\mathfrak{U})$ is asymptotically Poisson with parameter
\begin{equation}
\sum_{\sigma\in\Omega_n}\Qs([\sigma_{[n]}])\,
\1{H_n(h,\sigma)-A_n(h)\in\mathfrak{U}}.
\end{equation}
This quantity coincides with the kernel $\mathbf{K}_n(h,\mathfrak{U})$ defined in \eqref{eq:Kconve}, for $C_n(h)=nc=n\rho (\log(2)-\log(1+|m|))$.

Using the Laplace transform of a Poisson distribution, we obtain
\begin{equation}
\E_U\left[e^{-\int_{\R} f(x) \Hn(dx)}\right]
\simeq
e^{\int_{\R}(e^{-f(x)}-1)\mathbf{K}_n(h,dx)},
\end{equation}
where $\E_U$ denotes expectation with respect to the thinning variables. This yields the approximation
\begin{equation}
\E\[\E_U\left[e^{-\int_{\R} f(x) \Hn(dx)}\right]\]
\simeq
\E\left[
\exp\left\{\int_{\R}(e^{-f(x)}-1)\mathbf{K}_n(h,dx)\right\}
\right].
\end{equation}
Therefore, by the dominated convergence theorem, if the kernels $\mathbf{K}_n$ converge vaguely to the exponential measure $\D_{\tilde\lambda}$, by Lemma \ref{lem:DDD}, $\Hn$ converges in distribution to a Poisson point process follows. The vague convergence of $\mathbf{K}_n$ is precisely the content of Theorem \ref{them:0}.

The main difficulty is thus proving Theorem \ref{them:0}. 
The idea is to approximate the conditional distribution of $H_n(h,\sigma)$ via large deviation theory. For $A_n(h)=n\tilde a+o(n)$, Cram\'er's theorem gives
\begin{equation}
\mathbf{K}_n(h,[x,\infty))
=
e^{nc}\,
\Ps\big(H_n(h,\sigma)-A_n(h)\ge x\big)
\simeq
e^{nc-n G^*(\tilde a+n^{-1}x)+o(n)}.
\end{equation}
Using a Taylor expansion of $G^*$ around $\tilde a$ and the relation $G^*(\tilde a)=c$ and $\partial_aG^*( a)|_{a=\tilde a}=\tilde \lambda$, we obtain heuristically
\begin{equation}
\mathbf{K}_n(h,[x,\infty))
\simeq
e^{-\tilde \lambda x + o(n)}.
\end{equation}
If we can control the $o(n)$ correction in order to make it equal to $-\log(\tilde \lambda)+o(1)$, we get
\been{
\mathbf{K}_n(h,[x_1,x_2])\simeq \frac{e^{o(1)}}{\tilde \lambda}\(e^{-\tilde \lambda x_1}-e^{-\tilde \lambda x_2}\)\simeq e^{o(1)}\int^{x_2}_{x_1}\D_{\tilde \lambda}(dx)
}
for any $-\infty<x_1\leq x_2<\infty$, proving the vague convergence to the exponential measure.

To achieve this level of precision, standard large deviation estimates are not sufficient. We therefore rely on sharp large deviation results for random weighted sums of i.i.d.\ variables, as developed in \cite{Bovier}, which allow us to control the subexponential corrections beyond the leading rate function. This constitutes the core technical part of the paper.

\subsection{Organization of the paper}

The manuscript is organized as follows. The next subsection introduces the notation used throughout the paper. Section \ref{sec:2} proves Theorem \ref{them:1}, assuming Theorem \ref{them:0}. Sections \ref{sec:4} and \ref{sec:5} develop the sharp large deviation analysis and contain the proof of Theorem \ref{them:0}, which forms the technical core of the work.

}
\subsection{Main notation}
We denote by $\bar{R}$ the set of extended real numbers. We also define
\been{
\R_{> 0}:=(0,\infty),\quad \R_{\geq 0}:=[0,\infty),\quad \bar\R_{> 0}:=(0,\infty],\quad \bar\R_{\geq 0}:=[0,\infty].
}
An \emph{extended real function} is a function that takes values on $\bar{R}$ (or its subset). We say that an extended real function is continuous if
\been{
\limsup_{x\to x_0}f(x)=\liminf_{x\to x_0}f(x)=f(x_0),\quad \forall x_0\in A,
}
where both limits and $f(x_0)$ can be $\infty$ or $-\infty$. With this notation, if $f$ is an extended real-valued continuous function, then the set $\{x:\,f(x)<\infty\}$ is open in $A$.

We use the notation $\log(0)=-\infty$ and $\sign(0)=0$.


In the following, $\E[\cdot]$ denotes the expectation with respect to the random variables $h$.
\section{Proof of Theorem \ref{them:1}}
\label{sec:2}
\noindent
In this section, we prove Theorem \ref{them:1} assuming the validity of Theorem \ref{them:0}. We will proceed by computing the Laplace transform of the point processes $\mathbf{H}_n$.

Given a locally finite random measure $\mathbf{R}$ on $\R$, its Laplace transform is the functional defined on the set of all measurable non-negative functions $f:\R\to \bar\R_{\geq 0}$ by
\been{
\label{eq:RLaplace_trasf}
\mathcal{L}_{\mathbf{R}}(f):=\E_{\mathbf{R}}\[\exp\(-\int^{\infty}_{-\infty}f(x)\mathbf{R}(dx)\)\],
}
where, here, the operator $\E_{\mathbf{R}}$ denotes the expectation with respect to the randomness of the measure $\mathbf{R}$.

The distribution of a locally finite random measure is uniquely determined by its Laplace transform evaluated on the class of bounded, continuous, non-negative functions with compact support \cite{Kallenberg}. Therefore, to prove Theorem \ref{them:1}, it suffices to show that the sequence of Laplace transforms
\been{
(\mathcal{L}_{\mathbf{H}_n}(f))_{n\in \N},
}
converges to $\mathcal{L}_{\PPP_{\tilde \lambda}}(f)$ for any $f$ over this particular class of functions.

We begin by providing the Laplace transform of $\PPP_{\tilde \lambda}$.
\begin{lem}
\label{lem:DDD}
For any measurable $f:\R\to \bar\R_{\geq 0}$
\been{
\mathcal{L}_{\PPP_{\tilde \lambda}}(f)=\exp\(-\int^{\infty}_{-\infty}\(1-e^{-f(x)}\)\D_{\tilde \lambda}(dx)\).
}
\end{lem}
\begin{proof}
The Laplace transform of a PPP is well known in the theory of random measures (see, e.g., \cite{Kallenberg}).
\end{proof}
We now study the measure $\Qs$.
\begin{lem}
\label{lem:Qbounds}
There exists some $\delta>0$ such that, for any $\tau\in \{-1,1\}^n$,
\been{
\label{eq:boundQ}
\Qs([\tau])\leq e^{-n\delta},
}
\end{lem}
\begin{proof}
Let $1+ |m|\geq  1+m\tau_i$,
\been{
\Qs([\tau])=\(\frac{2}{1+|m|}\)^{n\rho}\prod^n_{i=1}\frac{1+m\tau_i}{2}\leq \(\frac{1+|m|}{2}\)^{n(1-\rho)}\prod^n_{i=1}\frac{1+m\tau_i}{1+|m|}\leq \(\frac{1+|m|}{2}\)^{n(1-\rho)}.
}
Finally, since $m\in (-1,1)$, $(1+|m|)/2\in (0,1)$. Thus we take
\been{
\delta =(1-\rho)\(\log(2)-\log(1+|m|)\)>0.
}
This completes the proof.
\end{proof}
We now compute the Laplace transform of the point process $\Hn$.
\begin{lem}
\label{lem:Lap_H}
For any measurable $f:\R\to \bar\R_{\geq 0}$
\been{
\mathcal{L}_{\Hn}(f)=\E\[\exp\(\sum_{\sigma \in \Omega_n}\log\(1+ \Qs([\sigma_{[n]}])\(e^{-f(H_n(h,\sigma)-A_n(h))}-1\)\)\)\]
}
\end{lem}
\begin{proof}
By definition, we have
\been{
\int^{\infty}_{-\infty}f(x)\Hn(dx)=\sum_{\sigma\in \Omega_n}f(H_n(h,\sigma)-A_n(h))\1{U_{\sigma}\leq \Qs([\sigma_{[n]}])}
}
Consequently, denoting by $\E_U$ the expectation over the uniform random variables $(U_{\sigma})_{\sigma \in \Omega_n}$, 
\been{
\aled{
\mathcal{L}_{\Hn}(f)&=\E\[\E_U\[\exp\(-\sum_{\sigma\in \Omega_n}f(H_n(h,\sigma)-A_n(h))\1{U_{\sigma}\leq \Qs([\sigma_{[n]}])}\)\]\]\\
&=\E\[\prod_{\sigma \in \Omega_n}\E_U\[e^{-f(H_n(h,\sigma)-A_n(h))\1{U_{\sigma}\leq \Qs([\sigma_{[n]}])}}\]\],
}
}
where we used the fact that the variables $U_{\sigma}$ are independent. Using the equality $e^{a\bm{1}}=1+\bm{1}(e^a-1)$ for any $\bm{1}\in \{0,1\}$, we have
\been{
\aled{
\mathcal{L}_{\Hn}(f)&=\E\[\prod_{\sigma \in \Omega_n}\E_U\[e^{-f(H_n(h,\sigma)-A_n(h))\1{U_{\sigma}\leq \Qs([\sigma_{[n]}])}}\]\]\\
&=\E\[\prod_{\sigma \in \Omega_n}\E_U\[1+\1{U_{\sigma}\leq \Qs([\sigma_{[n]}])}\(e^{-f(H_n(h,\sigma)-A_n(h))}-1\)\]\]\\
&=\E\[\prod_{\sigma \in \Omega_n}\(1+\E_U\[\1{U_{\sigma}\leq \Qs([\sigma_{[n]}])}\]\(e^{-f(H_n(h,\sigma)-A_n(h))}-1\)\,\)\]
}
}
By \eqref{eq:boundQ}, $\Qs([\sigma_{[n]}])\in [0,1]$; so $\E_U\[\1{U_{\sigma}\leq \Qs([\sigma_{[n]}])}\]=\Qs([\sigma_{[n]}])$, completing the proof.
\end{proof}
We are now ready to prove the theorem.
\begin{proof}[Proof of Theorem \ref{them:1}]
Define
\been{
c=\rho(\log(2)-\log(1+|m|)),\quad C_n(h):=n c.
}
Since $\rho\in (0,1)$
\been{
c< \log(2)-\E[\log(1+m\sign(h_1))]=\gamma,\quad C_n(h)<n\log(2)-\sum^n_{i=1}\log(1+m\sign(h_i))=\Gamma_n(h).
}
Hence, by Theorem \ref{them:0}, there exists $h-$measurable a sequence $(A_n)_{n\in \N}$ such that the sequence of measure kernels $(\mathbf{K}_n)_{n\in \N}$, defined by
\been{
\aled{
\mathbf{K}_n(h,\mathfrak{U})&=\Qs(\{\sigma:\,H_n(h,\sigma)-A_n(h)\in \mathfrak{U}\})\\
&=e^{n\,c}\Ps(\{\sigma:\,H_n(h,\sigma)-A_n(h)\in \mathfrak{U}\}),\quad h\in \R^{\N},\,\mathfrak{U}\in \mathcal{B}(\R)
}
}
converges vaguely, $\Gas$, to the deterministic measure $\D_{\tilde{\lambda}}$, with $\tilde{\lambda}$ defined in \eqref{eq:solac0}.

For any $x\in [-1,0]$ and $\alpha\in [0,1)$, we have
\been{
\frac{\alpha}{1-\alpha}x\leq \log(1+\alpha x)\leq \alpha x.
}
Consequently, for any measurable $f:\R\to \bar{R}_{\geq 0}$,
\been{
\aled{
\sum_{\sigma\in \Omega_n}\log\(1+ \Qs([\sigma_{[n]}])\(e^{-f(H_n(h,\sigma)-A_n(h))}-1\)\)&\leq \sum_{\sigma\in \Omega_n}\Qs([\sigma_{[n]}])\(e^{-f(H_n(h,\sigma)-A_n(h))}-1\)\\
&=\int^{\infty}_{-\infty}\(e^{-f(x)}-1\)\mathbf{K}_n(h,dx),
}
}
and
\been{
\aled{
\sum_{\sigma\in \Omega_n}\log\(1+ \Qs([\sigma_{[n]}])\(e^{-f(H_n(h,\sigma)-A_n(h))}-1\)\)&\geq \sum_{\sigma\in \Omega_n}\frac{\Qs([\sigma_{[n]}])}{1-\Qs([\sigma_{[n]}])}\(e^{-f(H_n(h,\sigma)-A_n(h))}-1\)\\
&\geq \frac{1}{1-e^{-n\delta}}\sum_{\sigma\in\Omega_n}\Qs([\sigma_{[n]}])\(e^{-f(H_n(h,\sigma)-A_n(h))}-1\)\\
&=\frac{1}{1-e^{-n\delta}}\int^{\infty}_{-\infty}\(e^{-f(x)}-1\)\mathbf{K}_n(h,dx),
}
}
where the second inequality follows from \eqref{eq:boundQ}. Therefore, for any bounded, continuous, non-negative, and compactly supported function $f$,
\been{
\lim_{n\to \infty}\sum_{\sigma\in \Omega_n}\log\(1+ \Qs([\sigma_{[n]}])\(e^{-f(H_n(h,\sigma)-A_n(h))}-1\)\)=\int^{\infty}_{-\infty}\(e^{-f(x)}-1\)\D_{\widetilde{\lambda}}(dx),\quad \gas
}
Combining this limit with Lemma \ref{lem:DDD} and \ref{lem:Lap_H}, and applying the Dominated Convergence Theorem, we obtain
\been{
\aled{
\lim_{n\to \infty}\mathcal{L}_{\Hn}(f)&=\E\[\lim_{n\to \infty}\exp\(\sum_{\sigma\in \Omega_n}\log\(1+ \Qs([\sigma_{[n]}])\(e^{-f(H_n(h,\sigma)-A_n(h))}-1\)\)\)\]\\
&=e^{\int^{\infty}_{-\infty}\(e^{-f(x)}-1\)\D_{\widetilde{\lambda}}(dx)}=\mathcal{L}_{\PPP_{\widetilde{\lambda}}}(f),
}
}
for any bounded, continuous, and compactly supported function $f$, completing the proof.
\end{proof}

\section{Sharp large deviation bound at finite $n$}
\label{sec:4}
\noindent
In this section, we develop an approximation, precise up to errors of order $o(1)$, for the probability 
\been{
\Ps(\{\sigma:\,H_n(h,\sigma)>a\}),\quad a>0,
}
in the regime of large but finite $n$. Our approach relies on the \emph{Strong (local) Large Deviation Principle} (SLDP), a refined version of the classical Large Deviation Principle (LDP).

The standard LDP (G\"artner--Ellis Theorem \cite[Theorem~2.3.6]{Dembo}) describes the exponential decay of rare-event probabilities through a rate 
function given by the Fenchel--Legendre transform ($\FLT$) of the Moment Generating Function ($\MGF$).

Although the rate function determines the leading exponential asymptotics, it cannot give a 
precise estimate of the probability itself, since it contains no information on the 
subleading corrections of order $O(1)$.

The SLDP refines the LDP by controlling subleading corrections of order $o(n)$ beyond the
leading exponential term. Bahadur and Ranga Rao established the standard SLDP for a sum of i.i.d. random variables \cite{Bahadur}. This result was later extended to general sequences of random variables by Chaganty and Sethuraman \cite{Chaganty}.

In this section, we use the version by Bovier and Mayer, who developed a conditional strong large deviation principle that provides the asymptotic approximation of the tail probability of weighted sums of i.i.d. random variables, conditionally on the i.i.d. random weights \cite[Theorem 1.6]{Bovier}. In our setting, the random i.i.d. variables are the spin components $\sigma_1,\sigma_2,\cdots $ and the random weights are the field components $h_1,h_2,\cdots$.

Throughout this section, we fix $\e \in (0, \tfrac12)$ and consider
\been{
\label{eq:mcond}
m \in [-1 + 2\e, 1 - 2\e].
}
We denote by $\langle \cdot \rangle$ the expectation over $\sigma$ with respect to the probability measure $\Ps$, conditionally on $h$. The $\MGF$ of $H_n$, conditionally on $h$, is defined as
\been{
M_n(h,\lambda):=\log\langle e^{\lambda H_{n}}\rangle=\log\( \prod^n_{i=1}\sum_{\sigma_i\in\{-1,1\}}\frac{1+m\sigma_i}{2}e^{\lambda h_i(\sigma_i-m)}\)=\sum^n_{i=1}g(\lambda h_i)
}
with
\been{
\label{eq:gm2}
g(\lambda):=\log\langle e^{\lambda (\sigma_1-m)}\rangle=\log\(\frac{1+m}{2}e^{\lambda (1-m)}+\frac{1-m}{2} e^{-\lambda (1+m)}\).
}
We denote by $M^*_{n}(h,\cdot)$ the $\FLT$ of $M_{n}(h,\cdot)$:
\been{
M^*_{n}(h,a):=\sup_{\lambda\in \R}(\lambda a-M_{n}(h,\lambda)).
}
We also define
\been{
\label{eq:defJGamma}
\Sigma_{n}(h):=-m\sum^n_{i=1}h_i+\sum^n_{i=1}|h_i|,\qquad \Gamma_{n}(h):=-\sum^n_{i=1}\log\(1+\sign(h_i)m\)+n\log(2),
}
and the set
{
\been{
\label{eq:defL}
\mathfrak{L}^{\e}_{n}:=\left\{h\in \R^{\N}:\quad 2\pi\sum^n_{i=1}h^2_i\leq n^{3/2},\quad 2\pi\min_{1\leq i\leq n} h^2_i\geq \frac{16}{n^5\e^8},\quad \Sigma_n(h)\in (n^{4/5}, n^{3/2}) \right\},
}
}
We denote by $M'_n(h,\cdot)$ and $M''_n(h,\cdot)$ the derivatives with respect to the second argument, keeping $h$ fixed.
\begin{proposition}
\label{prop:finiteREM}
Assume that Assumption \ref{assum:1} holds, and fix $\lambda^*>1$ and $\e\in (0,\tfrac{1}{2})$. There exists $N_{\e}>0$ such that, for any $n>N_{\e}$, $\Gae$ $h \in \mathfrak{L}^{\e}_n$, and any
\been{
\label{eq:Ccond}
C \in (\e \Gamma_n(h), (1-\e)\Gamma_n(h)),
}
there exists a set $\mathfrak{R}_{n,\lambda^*}(h)$ such that the following statements hold:
\begin{itemize}
\item for any $x\in \mathfrak{R}_{n,\lambda^*}(h)$
\been{
\label{eq:LBP}
\Ps(\{\sigma:\,H_n(h,\sigma)\geq \tilde A_n(h)+x\})\leq \sqrt{\frac{ M''_n(h,\tilde\Lambda_n(h))}{ M''_n(h,\tilde\Lambda^{x}_n(h))}} \frac{1}{\tilde \Lambda^{x}_n(h)}e^{-C-\tilde \Lambda_n(h)x }(1+o(1));
}
\item for any $x\in \mathfrak{R}_{n,\lambda^*}(h)$
\been{
\label{eq:UBP}
 \Ps(\{\sigma:\,H_n(h,\sigma)\geq \tilde A_n(h)+x\})\geq \sqrt{\frac{ M''_n(h,\tilde\Lambda_n(h))}{ M''_n(h,\tilde\Lambda^{x}_n(h))}}\frac{1}{\tilde \Lambda^{x}_n(h)}e^{-C-\tilde \Lambda^{x}_n(h)x }(1+o(1)).
}
\end{itemize}
Here $\tilde A_n(h)$, $\tilde \Lambda_n(h)$ and $\tilde \Lambda^{x}_n(h)$ are $h-$measurable random variables such that, if $h \in \mathfrak{L}^{\e}_n$ and $x\in \mathfrak{R}_{n,\lambda^*}(h)$, they are solutions of the following coupled equations
\been{
\label{eq:Gstareq}
\begin{cases}
\aled{
&M^*_n(h, \tilde A_n(h))+\frac{1}{2}\log(2\pi M''_n(h,\tilde\Lambda_n(h)))=C\quad \textup{and}\quad \tilde A_n(h)\in (0, \Sigma_{n}(h)),\\
&M'_n(h,\tilde \Lambda_n(h))= \tilde A_n(h),\\
&M'_n(h,\tilde \Lambda^{x}_n(h))= \tilde A_n(h)+x.
}
\end{cases}
}
\end{proposition}
\begin{remark}
We postpone the precise definition of the set $\mathfrak{R}_{n,\lambda^*}(h)$ to Lemma \ref{lem:nonempty}. The definition of the set $\mathfrak{R}_{n,\lambda^*}(h)$ is crucial for the proof of Theorem \ref{them:1}.
\end{remark}
The proof of the above proposition relies on several intermediate lemmas, which we present in separate subsections. In the next subsection, we present the Bovier-Mayer SLDP result for our model, which constitutes the basis for the above proposition. In the following subsection, we state the analytical properties of the function $M$ and $M^*$. Hence, we establish the existence of solutions to the system of equations \eqref{eq:Gstareq}. The section concludes with the proof of the proposition.
\subsection{The Bovier-Mayer SLDP}
In this subsection, we present the main mathematical tool of this section, namely the Bovier-Mayer SLDP results. We show that $H_n$ satisfies the required assumptions and then state the corresponding SLDP specialized to our setting.

To this end, we first derive quantitative estimates for the function $g$, defined in \eqref{eq:gm2}, and its derivatives. A direct computation yields
\been{
\label{eq:ders}
g'(\lambda)=\frac{(1+m) e^{\lambda (1-m)}-(1-m)e^{-\lambda(1+m)}}{(1+m) e^{\lambda (1-m)}+(1-m)e^{-\lambda(1+m)}}-m,\quad g''(\lambda)=1-\left(g'(\lambda)+m\right)^2.
}
The following lemma provides essential bounds on the function $g$ and its derivatives, which will be needed for the application of the Bovier-Mayer SLDP.
\begin{lem}
\label{lem:bounds}
 We have 
 \been{
 \label{eq:g2trivial}
 g(0)=g'(0)=0,\quad g''(0)=1-m^2.
 }
 For  any $\lambda\in \mathbb{R}$
\been{
\label{eq:g2UB}
0\leq g(\lambda )\leq 2|\lambda|,\quad 0\leq \sign(\lambda)g'(\lambda )\leq 2, \quad g''(\lambda )\in (0,1).
}
and
\been{
\label{eq:g2eq}
g''(\lambda)\geq (1-m^2)e^{-2|\lambda|}
}
Moreover, for any $\lambda,\,\lambda'\in \R$
\been{
\label{eq:Lip}
|g''(\lambda )-g''(\lambda')|\leq 2 |\lambda-\lambda'|.
}
Finally, given $h_1\in \R$ and $\lambda\in \R$
\been{
\label{eq:g2}
g(\lambda h_1)\leq \log\(\frac{1+\sign(h_1)m}{2}\)+\frac{1-\sign(h_1)m}{1+\sign(h_1)m}e^{-2\lambda|h_1|}+\lambda(|h_1|-m h_1)
}
\end{lem}
\begin{proof}
A direct computation from \eqref{eq:gm2} and \eqref{eq:ders} gives \eqref{eq:g2trivial}. By the Jensen inequality
\been{
g(\lambda )=\log\langle e^{\lambda (\sigma_1-m)}\rangle\geq \lambda\langle (\sigma-m)\rangle=0.
}
For a measurable function $f:\{-1,1\}\to \R$ let
\been{
\langle f(\sigma_1)\rangle_{\lambda}:=\frac{\langle e^{\lambda(\sigma_1-m)}f(\sigma)\rangle}{\langle e^{\lambda(\sigma_1-m)}\rangle}.
}
Thus, 
\been{
\label{eq:ders12}
g'(\lambda )=\langle \sigma_1\rangle_{\lambda}-m,\quad g''(\lambda )=\langle (\sigma_1-m)^2\rangle_{\lambda}-\langle \sigma_1-m\rangle^2_{\lambda}=1-\langle\sigma_1\rangle^2_{\lambda}
}
We have
\been{
g'(\lambda )+m=\langle \sigma_1\rangle_{\lambda}=\frac{(1+m) e^{\lambda (1-m)}-(1-m)e^{-\lambda(1+m)}}{(1+m) e^{\lambda (1-m)}+(1-m)e^{-\lambda(1+m)}}\in (-1,1),\quad \forall \lambda\in \R.
}
Thus $g''(\lambda )\in (0,1)$ and
\been{
|g'(\lambda )|\leq |m|+|\langle \sigma_1\rangle_{\lambda}|\leq 2.
}
Moreover, since $g''(\lambda)>0$ and $g'(0)=0$,
\been{
\sign( g'(\lambda ))=\sign (\lambda)\Longrightarrow \sign (\lambda)g'(\lambda )=|g'(\lambda )|\leq 2
}
Since $g(0)=0$, the upper bound on the first derivative implies
\been{
\label{eq:intermediateg}
g(\lambda)+m\lambda\leq |\lambda|\sup_{\lambda \in \R}|g'(\lambda )+m|=|\lambda|.
}
Thus $g(\lambda)\leq |\lambda|+m\lambda\leq 2|\lambda|$. We now compute the third derivative. The equivalences \eqref{eq:ders12} and the proved bounds \eqref{eq:g2UB} gives
\been{
g'''(\lambda)=-\frac{d}{d\lambda}\langle \sigma_1\rangle^2_{\lambda}=-\frac{d}{d\lambda}(g'(\lambda)+m)^2=-2(g'(\lambda)+m)g''(\lambda)\overset{\eqref{eq:g2UB}}{\in} (-2,2).
}
proving the Lipschitz constant \eqref{eq:Lip}. Moreover, solving the differential equation, we get
\been{
\frac{d}{d\lambda}(\log g''(\lambda)) =\frac{g'''(\lambda)}{g''(\lambda)}=-2\frac{d}{d\lambda}(g(\lambda)+m\lambda)\Longrightarrow g''(\lambda)=g''(0)e^{-2g(\lambda)-2\lambda m}=(1-m^2)e^{-2g(\lambda)-2\lambda m}.
}
So, by \eqref{eq:intermediateg},
\been{
g''(\lambda)=(1-m^2)e^{-2g(\lambda)-2\lambda m}\geq (1-m^2)e^{-2|\lambda|}.
}
Finally
\been{
\aled{
g(\lambda h_1)&=\log\(\frac{1+m}{2}e^{\lambda h_1(1-m)}+\frac{1-m}{2}e^{-\lambda h_1(1+m)}\)\\
&= \log\(\frac{1+\sign(h_1)m}{2}+\frac{1-\sign(h_1)m}{2}e^{-2\lambda|h_1|}\)+\lambda(|h_1|-m h_1)\\
&\leq \log\(\frac{1+\sign(h_1)m}{2}\)+\frac{1-\sign(h_1)m}{1+\sign(h_1)m}e^{-2\lambda|h_1|}+\lambda(|h_1|-m h_1)
}
}
\end{proof}
We now state the Bovier-Mayer SLDP for the model under consideration. For $a\in (0,\Sigma_{n}(h))$, let us now define
\been{
J_n(h,a):=\frac{1}{\sqrt{2\pi M''_n(h,\Lambda_n(h,a))}\Lambda_n(h,a)}e^{- M^*_n(h,a) }.
}
\begin{lem}[Bovier-Mayer SLDP]
\label{lem:Bovier}
Fix $\lambda^*\in \R_{>0}$ and let $\varsigma^*:=\E[h_1 g'(h_1\lambda^*)]$. Given $n\in \N$,
\been{
\label{eq:Bov}
\Ph\(\forall a\in (0,n \varsigma^*\wedge\Sigma_n(h))\quad \Ps(\{\sigma:\,H_n(h,\sigma)\geq a \})=J_n(h,a)(1+o(1))\)=1.
}
\end{lem}
\begin{proof}
Using the inequality \eqref{eq:g2UB}, we get
\been{
g(\lambda)\leq 2|\lambda|<\infty,\quad
|g'(\lambda)|\leq 2,\quad |g''(\lambda)|\leq 1, \quad \forall \lambda\in \R.
}
Thus, since by Assumption \ref{assum:1} $\E[|h_1|]=\psi_3$, we have
\been{
\E[g(h_1\lambda)]\leq 2|\lambda| \psi_3,\quad \E[h_1g'(\lambda h_1)]<2\E[|h_1|]\leq 2\psi_3
}
and $h^2_ig''(\lambda h_i)\leq h^2_i$ for any $\lambda\in \R$. Thus, all the hypotheses of the Bovier-Mayer strong large deviation result (\cite[Theorem 1.6]{Bovier}) are satisfied, and \eqref{eq:Bov} holds for any
\been{
a\in(n\E[h_1]\langle \sigma_1-m\rangle,n\E[h_1 g'(h_1\lambda^*)])= (n\E[h_1]\langle \sigma_1-m\rangle,n\varsigma^*),}
where
\been{
\E[h_1]\langle \sigma_1-m\rangle=0.
}
So, for any $a\in (0,n\varsigma^*)$, the Bovier-Mayer strong large deviation result applies.
\end{proof}

\subsection{Analytical properties of $M_n$ and $M^*_n$}
In this subsection, we enumerate all the relevant properties of the $\MGF$ and its $FLT$.

A direct computation yields
\been{
\label{eq:ders0}
M'_n(h,\lambda)=\sum^n_{i=1}h_ig'(\lambda h_i),\qquad M''_n(h,\lambda)=\sum^n_{i=1}h^2_ig''(\lambda h_i).
}
The random variable $H_n$, conditionally on $h$, has mean $0$
\been{
\label{eq:subga1}
\langle H_{n}\rangle= \sum^n_{i=1}h_i\sum_{\sigma_i\in {-1,1}}\frac{1+m\sigma_i}{2}(\sigma_i-m)=0,
}
and 
\been{
\langle e^{|\lambda||H_{n}|}\rangle\leq e^{2|\lambda|\sum^n_{i=1}|h_i|}.
}
Moreover, if $h\neq 0$, then $H_n$ is not constant.

We denote by $\dot M^*_n(h,\cdot)$ and $\ddot M^*_n(h,\cdot)$ the derivatives of $M^*_n(h,\cdot)$ with respect to the second argument, keeping $h$ fixed. 
\begin{lem}[Analytical properties of $M_n$ and $M^*_n$]
\label{lem:exists}
 Fix $h\in \mathbb{R}^n\setminus \{0\}$. Then $\MGF$ $M_{n}(h,\cdot)$ is continuous, infinitely differentiable, and verifies the following properties
 \begin{enumerate}
 \item $M_{n}(h,0)=M'_{n}(h,0)=0$;
 \item $M''_n(h,\lambda)>0$ for any $\lambda\in \R$;
 \item $\{M'_n(h,\lambda):\, \lambda \in \R_{> 0}\}= (0,\Sigma_n(h))$;
 \item there exists a continuous increasing function $\Lambda_n(h,\cdot ):[0,\Sigma_n(h))\to \R_{\geq 0}$ such that
\begin{equation}
 \label{eq:Mstat}
M'_{n}(h,\Lambda_n(h,a))=a;
 \end{equation}
Moreover $\Lambda_n(h,0)=0$ and $\Lambda_n(h,(0,\Sigma_n(h)))=\R_{> 0}$.
\end{enumerate}
The $\FLT$ $M^*_n(h,\cdot)$ satisfies
 \begin{enumerate}
 \setcounter{enumi}{4}
 \item for any $a\in (0,\Sigma_{n}(h))$
 \begin{equation}
 \label{eq:Mstarstat}
M^*_{n}(h,a)=\Lambda_n(h,a) a-M_{n}(h,\Lambda_n(h,a)),\qquad \dot M^*_n(h,a)=\Lambda_n(h,a),
 \end{equation}
 and
 \begin{equation}
 \label{eq:Mstarstat2}
 \ddot M^*_n(h,a)=\dot \Lambda_n(h,a)=\frac{1}{M''_n(h,\Lambda_n(h,a))};
 \end{equation}
 \item $M^*_n(h,\cdot)$ is strictly increasing in $[0,\Sigma_{n}(h))$;
 \item $M^*_n(h,(0,\Sigma_{n}(h))\,)=(0,\Gamma_{n}(h))$ and $M^*_n(h,\Sigma_{n}(h)\,)=\Gamma_{n}(h)$;
 \item there exists a continuous increasing function $A_n(h,\cdot):(0,\Gamma_n(h))\to (0,\Sigma_n(h))$ such that
 \been{
 \label{eq:GA}
 M^*_n(h,A_n(h,c))=c,\quad \forall c\in (0,\Gamma_n(h)).
 }
 \end{enumerate}
\end{lem}
\begin{proof}
Since $\lambda\mapsto g(\lambda)$ is continuous and infinitely differentiable, the function $M_n$ is continuous and infinitely differentiable. We prove the remaining properties separately.
\begin{proof}[Proof of Claim $(1)$]
\begin{equation}
M_n(h,0)=\log\langle 1 \rangle=0,\qquad 
M'_{n}(h,0)=\langle H_n \rangle\overset{\eqref{eq:subga1}}{=}0.
\end{equation}
\end{proof}
\begin{proof}[Proof of Claim $(2)$]
By Lemma \ref{lem:bounds} and the formula \eqref{eq:ders0}, if $h\neq 0$, then $M''_n(h,\lambda)>0$.
\end{proof}
\begin{proof}[Proof of Claim $(3)$]
Since $M''_n(h,\lambda)>0$ for any $\lambda \in \R$, $M'_n(h,\cdot)$ is a strictly increasing function and it is continuous. Thus
\been{
\{M'_n(h,\lambda):\, \lambda \in \R_{> 0}\}=\(M'_n(h,0),\lim_{\lambda\to \infty}M'_n(h,\lambda)\)=\(0,\lim_{\lambda\to \infty}M'_n(h,\lambda)\).
}
We have
\been{
\lim_{\lambda\to \infty }h_1g'(\lambda h_1)=\lim_{\lambda\to \infty }h_1\frac{(1+m) e^{\lambda h_1(1-m)}-(1-m)e^{-\lambda h_1(1+m)}}{(1+m) e^{\lambda h_1(1-m)}+(1-m)e^{-\lambda h_1(1+m)}}-h_1m=|h_1|-m h_1.
}
So
\been{
\lim_{\lambda\to \infty }M'_n(h,\lambda)=\lim_{\lambda\to \infty }\sum^n_{i=1}h_ig'(\lambda h_i)=\sum^n_{i=1}|h_i|-m\sum^n_{i=1}h_i=\Sigma_n(h).
}
\end{proof}
\begin{proof}[Proof of Claim $(4)$]
The Claim $(1)$, $(2)$ and $(3)$ of this Lemma imply that the restriction $M'_n(h,\cdot):\R_{\geq 0}\to [0,\Sigma_n(h))$ is invertible. Thus, we define
$\Lambda_n(h,\cdot) := (M'_n(h,\cdot))^{-1} : [0,\Sigma_n(h)) \to \R_{\geq 0}$, which is continuous and strictly increasing, since $M'_n$ is continuous and strictly increasing. By definition, $\Lambda_n(h,a)$ is the unique solution in $\mathbb{R}_{\geq 0}$ of \eqref{eq:Mstat}. Moreover, since $M'_n(h,0)=0$, $\Lambda_n(h,0)=0$, and, since $\Lambda_n(h,\cdot)$ is strictly increasing, $\Lambda_n(h,a)>0$ for $a>0$.
\end{proof}
\begin{proof}[Proof of Claim $(5)$]
Since the function $\lambda \mapsto \lambda a-M_n(h,\lambda)$ is strictly concave, the stationary point is also the supremum. Moreover, since $M''_n(h,\lambda)>0$ for any $\lambda>0$, by the Implicit Function Theorem, the function $a\mapsto \Lambda_n(h,a)$ is differentiable, with
\been{
\dot \Lambda_n(h,a)=\frac{1}{M''_n(h, \Lambda_n(h,a))}.
}
Thus:
\been{
\dot M^*_n(h,a)=\dot \Lambda_n(h,a)\(a-M'_n(h, \Lambda_n(h,a))\)+\Lambda_n(h,a)=\Lambda_n(h,a).
}
\end{proof}
\begin{proof}[Proof of Claim $(6)$]
Claim $(4)$ and Claim $(5)$ of this Lemma prove the Claim.
\end{proof}
\begin{proof}[Proof of Claim $(7)$]
Since $M^*_n(h,\cdot)$ is a strictly increasing function and it is continuous, we have
\been{
\{M^*_n(h,a):\, a \in (0,\Sigma_n(h))\}=\(M^*_n(h,0),M^*_n(h,\Sigma_n(h))\).
}
By Claim $(1)$, Claim $(4)$ and Claim $(5)$ of this Lemma
\been{
M^*_n(h,0)=-M_n(h,\Lambda_n(h,0))=-M_n(h,0)=0.
}
For the other term, we have
\been{
\aled{
M^*_n(h,\Sigma_n(h))&=\sup_{\lambda\in \R}\sum^n_{i=1}\(\lambda|h_i|-\lambda h_i m-\log((1+m) e^{\lambda h_i(1-m)}+(1-m)e^{-\lambda h_i(1+m)})+\log(2)\)\\
&=\sup_{\lambda\in \R}\sum^n_{i=1}\(-\log((1+\sign(h_i)m)+(1-\sign(h_i)m)e^{-2\lambda|h_i|})\)+n\log(2).
}
}
The supremum is achieved at $\lambda\to \infty$. Thus
\been{
M^*_n(h,\Sigma_n(h))=n\log(2)-\sum^n_{i=1}\log(1+\sign(h_i)m)=\Gamma_n(h).
}
\end{proof}
\begin{proof}[Proof of Claim $(8)$]
The Claim $(6)$ and $(7)$ of this Lemma imply that the restriction
\been{
M^*_n(h,\cdot):(0,\Sigma_n(h))\to (0,\Gamma_n(h))
}
is invertible. Then, we define
\been{
A_n(h,\cdot) := (M^*_n(h,\cdot))^{-1} : (0,\Gamma_n(h)) \to(0,\Sigma_n(h)),
}
which is continuous and strictly increasing. By definition, $A_n(h,c)$ is the unique solution in $(0,\Sigma_n(h))$ of \eqref{eq:GA}. 
\end{proof}
\end{proof}
\subsection{Existence of the solution to \eqref{eq:Gstareq}}
We prove that there exists a solution $(\tilde A_n(h), \tilde \Lambda_n(h),\tilde \Lambda^{x}_n(h))$ to the system of equations \eqref{eq:Gstareq}. Throughout this subsection, we will always assume that
\been{
\label{eq:Ccond2}
C\in (\e\Gamma_n(h),(1-\e)\Gamma_n(h)).
}
We also recall that $m$ verifies \eqref{eq:mcond}.

By Lemma \ref{lem:exists} the functions $M'_n(h,\cdot)$ and $M^*_n(h,\cdot)$ are invertible over the appropriate range, and the inverses are given respectively by the functions $a\mapsto \Lambda_n(h,a)$ and $c\mapsto A_n(h,c)$, defined in Claim $(4)$ and Claim $(8)$ of that lemma.

Although the following function 
\been{
a\mapsto M^*_n(h,a)+\frac{1}{2}\log(2\pi M''_n(h,\Lambda_n(h,a)))
}
is not necessarily invertible, we will show that, if $h\in \mathfrak{L}^{\e}_n$, a solution to \eqref{eq:Gstareq} exists and can be approximated by the functions $A_n(h,\cdot)$ and $\Lambda_n(h,\cdot)$.

To this end, we need to analyze the behavior of the derivatives of $M$. Using Lemma \ref{lem:bounds}, we can establish upper and lower bounds for $M''_n(h,\lambda)$ for any $h\in \mathfrak{L}^{\e}_n$ and $\lambda\in \R$,
\begin{lem}
\label{lem:bounds2}
If $h\in \mathfrak{L}^{\e}_n$ and $\lambda>0$, then
\been{
\label{eq:g2LB}
 \frac{4}{n^6\e^4}\(\Gamma_n(h)+M_n(h,\lambda)-\lambda M'_n(h,\lambda)\)^2\leq 2\pi M''_n(h,\lambda)\leq n^{3/2}.
}
and
\been{
\label{eq:g2LB_lip}
  |M''_n(h,\lambda)-M''_n(h,\lambda')|\leq 2\(\sum^n_{i=1}|h^3_i|\)|\lambda-\lambda'|
}
\end{lem}
\begin{proof}
The Lipschitz bound \eqref{eq:Lip} gives
\been{
|M''_n(h,\lambda)-M''_n(h,\lambda')|\leq \sum^n_{i=1}h^2_i|g''(h_i\lambda)-g''(h_i\lambda')|\leq  2\(\sum^n_{i=1}|h^3_i|\)|\lambda-\lambda'|,
}
proving  \eqref{eq:g2LB_lip}. The upper bound \eqref{eq:g2UB} in Lemma \ref{lem:bounds} give
\been{
2\pi M''_n(h,\lambda)\leq 2\pi\sum^n_{i=1}h^2_i\leq n^{3/2},\quad \forall h\in \mathfrak{L}^{\e}_{n}.
}
We now prove the lower bound. 
By the lower bound \eqref{eq:g2eq}, the definition of $\mathfrak{L}^{\e}_n$, and the Jensen inequality give
\been{
\label{eq:M2ineq}
2\pi M''_n(h,\lambda)=2\pi\sum^n_{i=1}h^2_ig''(h_i\lambda)\overset{\eqref{eq:defL}}{\geq} 16\frac{1-m^2}{n^5\e^8}\sum^n_{i=1}e^{-2|h_i|\lambda}\geq  16\frac{1-m^2}{n^6\e^8}\(\sum^n_{i=1}e^{-|h_i|\lambda}\)^2.
}
We also have
\been{
\aled{
 M'_n(h,\lambda)&=\sum^n_{i=1}h_i\(\frac{\frac{1+m}{2}e^{\lambda h_i(1-m)}-\frac{1-m}{2}e^{-\lambda h_i(1+m)}}{\frac{1+m}{2}e^{\lambda h_i(1-m)}+\frac{1-m}{2}e^{-\lambda h_i(1+m)}}-m\)\\
 &=\Sigma_n(h)-\sum^n_{i=1}\frac{(1-\sign(h_i)m)|h_i|e^{-2\lambda |h_i|}}{\frac{1+\sign(h_i)m}{2}+\frac{1-\sign(h_i)m}{2}e^{-2\lambda |h_i|}}\geq \Sigma_n(h)-2\sum^n_{i=1}\frac{1+|m|}{1-|m|}|h_i|e^{-2\lambda |h_i|}
}
}
Since $m\in [-1+2\e,1-2\e]$ and $\e\in (0,\tfrac{1}{2})$
\been{
\label{eq:m_ineq}
\frac{1+|m|}{1-|m|}\leq \frac{1-m^2}{(1-|m|)^2}\leq \frac{\sqrt{1-m^2}}{\e^2}.
}
Thus
\been{
\label{eq:M2_uu}
\Sigma_n(h)-M'_n(h,\lambda)\leq 2\frac{\sqrt{1-m^2}}{\e^2}\sum^n_{i=1}|h_i|e^{-2\lambda |h_i|}.
}
Moreover, by the upper bound \eqref{eq:g2}
\been{
\label{eq:M2_uuu}
M_n(h,\lambda)\leq \lambda\Sigma_n(h)-\Gamma_n(h)+\sum^n_{i=1}\frac{1+|m|}{1-|m|}e^{-2\lambda |h_i|}\overset{\eqref{eq:m_ineq}}{\leq} \lambda\Sigma_n(h)-\Gamma_n(h)+\frac{\sqrt{1-m^2}}{\e^2}\sum^n_{i=1}e^{-2\lambda |h_i|}
}
Combining \eqref{eq:M2_uu} and \eqref{eq:M2_uuu}, we get
\been{
M_n(h,\lambda)\leq \lambda M'_n(h,\lambda)-\Gamma_n(h)+\frac{\sqrt{1-m^2}}{\e^2}\sum^n_{i=1}e^{-2\lambda |h_i|}(1+2|h_i|\lambda)
}
So, using the inequality $(1+2x)e^{-2x}\leq 2e^{-x}$ for any $x\geq 0$, we get
\been{
\label{eq:M2ineq2}
\Gamma_n(h)+M_n(h,\lambda)-\lambda M'_n(h,\lambda)\leq 2\frac{\sqrt{1-m^2}}{\e^2}\sum^n_{i=1}e^{-\lambda |h_i|}.
}
Combining \eqref{eq:M2ineq} and \eqref{eq:M2ineq2}
\been{
2\pi M''_n(h,\lambda)\geq \frac{4}{n^6\e^4}\(\Gamma_n(h)+M_n(h,\lambda)-\lambda M'_n(h,\lambda)\)^2
}
\end{proof}
Now, we give an estimate of $\Gamma_n(h)$ and $\Sigma_n(h)$.
\begin{lem}
\label{lem:insideC1}
For any $h\in \R^{\N}$
\been{
n\e\leq \Gamma_n(h)\leq -n\log(\e).
}
\end{lem}
\begin{proof}
Since $m\in [-1+2\e,1-2\e]$, the definition of $\Gamma_n(h)$ in \eqref{eq:defJGamma} gives
\been{
\Gamma_n(h)\geq n\log(2)-n\log(1+|m|)\geq -n\log(1-\e)\geq n\e,
}
and
\been{
\Gamma_n(h)\leq n\log(2)-n\log(1-|m|)\leq -n\log(\e).
}
\end{proof}
Now, define
\been{
\label{eq:defCpm}
C^{+}_n:=C+2\log(n),\qquad C^{-}_n:=C-\frac{3}{4}\log(n).
}
Note that $C^{-}_n<C<C^{+}_n$.

For large $n$, $\frac{1}{n}(C^{-}_n-C)$ and $\frac{1}{n}(C^{+}_n-C)$ are ``small''. Under the condition \eqref{eq:Ccond2}, we have the following.
\begin{lem}
\label{lem:insideC}
Fix $h\in \R^n$. There exists $N_{\e}>0$ (independent of $h$ and $C$) such that, for any $n>N_{\e}$ 
 \been{
 \label{eq:Cpminside}
 C^{-}_n>0,\quad \textup{and}\quad \Gamma_n(h)-C^{+}_n\geq \frac{\e^2}{2}n.
 }
\end{lem}
\begin{proof}
If $n>-\frac{\log(\e)}{\e^5}+100=:N_{\e}$, then $4\log(n)<\e^2n$. Thus, by Lemma \ref{lem:insideC1} and since $C$ verifies \eqref{eq:Ccond2}
\been{
\Gamma_n(h)-C^{+}_n= \Gamma_n(h)-C-(C^{+}_n-C)\geq \e \Gamma_n(h)- 2\log(n)\geq \frac{1}{2}\e^2 n,
}
and
\been{
C^{-}_n= C+ C^{-}_n-C\geq \e \Gamma_n(h)-\frac{3}{4}\log(n)\geq \e^2 n-\frac{3}{4}\log(n)\geq 0.
}
\end{proof}
Fixing $h\in \R^n$, the function $A_n(h,\cdot):(0,\Gamma_n(h))\to (0,\Sigma_n(h))$, introduced in Claim $(8)$ of Lemma \ref{lem:exists}, is well-defined, continuous and increasing. By the previous Lemma, both $C^{+}_n$ and $C^{-}_n$ are in $(0,\Gamma_n(h))$ for any $n>N_\e$. Hence, for any $n>N_\e$, we can define
\been{
\label{eq:defApm}
A^{+}_n(h):=A_n(h,C^{+}_n)\in (0,\Sigma_n(h)),\qquad A^{-}_n(h):=A_n(h,C^{-}_n)\in (0,\Sigma_n(h)). 
}
Moreover, since $C^{-}_n<C^{+}_n$, 
\been{
A^{-}_n(h)\leq A^{+}_n(h).
}
The idea is that $A^{-}_n(h)$ and $A^{+}_n(h)$ provide lower and upper bounds for the solution $\tilde A_n(h)$ of \eqref{eq:Gstareq}. Moreover, we will show in the next Lemma, for $n$ large enough, $\tfrac{1}{n}A^{-}_n(h)$ and $\tfrac{1}{n}A^{+}_n(h)$ are ``close''. This observation will be crucial also in the next section, where we will compute the limit $n\to \infty$ of $\tfrac{1}{n}\tilde A_n(h)$.
\begin{lem}
\label{lem:Adif}
For any $n>N_{\e}$, if $h\in \mathfrak{L}^{\e}_n$, then
\been{
|A^{+}_n(h)-A^{-}_n(h)|\leq n^{3/4}\sqrt{\log(n)}.
}
\end{lem}
\begin{proof}
Since $\Lambda_n(h,a)>0$ for any $a\in (0,\Sigma_n(h))$ and $A_n(h,c)\in (0,\Sigma_n(h))$ for any $c\in (0,\Gamma_n(h))$, the implicit function Theorem and Claim $(5)$ of Lemma \ref{lem:exists} give
\been{
\frac{\partial }{\partial c}A_n(h,c)=\frac{1}{\dot M^*_n(h,A_n(h,c))}=\frac{1}{\Lambda_n(h,A_n(h,c))}>0,\quad \forall c\in (0,\Gamma_n(h)).
}
Using again Claim $(5)$ of Lemma \ref{lem:exists} and the fact that $\Lambda_n(h,0)=0$, we have
\been{
\Lambda_n(h,a)\geq a\min_{a'\in [0,a]}\dot\Lambda_n(h,a)\geq a\min_{a'\in [0,a]}\frac{1}{M''_n(h,\Lambda_n(h,a'))},\quad \forall a\in (0,\Sigma_n(h))
}
Thus, by Lemma \ref{lem:bounds2}, if $h\in  \mathfrak{L}^{\e}_n$, then
\been{
A_n(h,c)\frac{\partial }{\partial c}A_n(h,c)\leq \max_{a'\in [0,A_n(h,c)]}M''_n(h,\Lambda_n(h,a'))\overset{\eqref{eq:g2LB}}{\leq} n^{3/2}(2\pi)^{-1}\leq \frac{1}{6} n^{3/2},\quad \forall h\in \mathfrak{L}^{\e}_n.
}
Hence
\been{
(A^{+}_n(h))^2-(A^{-}_n(h))^2= 2\int^{C^{+}_n}_{C^{-}_n}dcA_n(h,c)\frac{\partial }{\partial c}A_n(h,c)\leq  \frac{1}{3}n^{3/2}\(C^{+}_n(h)-C^{-}_n(h)\)\leq n^{3/2}\log(n)
}
and
\been{
(A^{+}_n(h))^2-(A^{-}_n(h))^2\geq (A^{+}_n(h)-A^{-}_n(h))^2.
}
Hence, the above two inequalities complete the proof.
\end{proof}
The previous lemma immediately implies the following.

\begin{lem}
\label{lem:nonempty}
Given $\lambda^*>1$ let 
\been{
\label{varsi}
\varsigma^*:=\E[h_1 g'(h_1\lambda^*)].
}
There exists $ N_{\e}>0$ (independent of $\lambda^*$) such that, for any $n>N_{\e}$ and $h\in \mathfrak{L}^{\e}_n$, the following set
\been{
\label{eq:def_R}
\mathfrak{R}_{n,\lambda^*}(h):=\left\{x\in \R:\quad -A^{-}_{n}(h)< x<(n\varsigma^*\wedge \Sigma_n(h))-A^{+}_{n}(h)\right\}
}
is nonempty. 
\end{lem}
\begin{proof}
By Lemma \ref{lem:bounds} $g'(0)=0$, $h_1g'(h_1\lambda^*)\geq 0$, and $g''(h_1\lambda^*)>0$. Hence $h_1g'(h_1\lambda^*)\geq h_1g'(h_1)\geq 0$ for any $\lambda^*>1$. So, by Assumption \ref{assum:1}, $h_1g'(h_1\lambda^*)\geq h_1g'(h_1)>0$ with probability higher than $0$. Consequently
\been{
\aled{
\varsigma^*&=\E[h_1 g'(h_1\lambda^*)]\geq \E[h_1 g'(h_1)]>k>0
}
}
where $k$ is some constant independent of $n$ and $\lambda^*$. Hence, by the definition of $\mathfrak{L}^{\e}_n$, $n>k^{-5}$
\been{
n\varsigma^*\wedge \Sigma_n(h)\geq n^{4/5},\quad \forall h\in \mathfrak{L}^{\e}_n.
}
Let $N_{\e}$ be the threshold number defined in Lemma \ref{lem:Adif}. If $n>2^{52}\vee k^{-5}\vee N_\e$, $n^{3/4}\sqrt{\log(n)}\leq n^{4/5}$, and the above inequality and Lemma \ref{lem:Adif} yields
\been{
A^{+}_n(h)-A^{-}_n(h)\leq n^{3/4}\sqrt{\log(n)}<n^{4/5}<n\varsigma^*\wedge \Sigma_n(h),\quad \forall h\in \mathfrak{L}^{\e}_n.
}
\end{proof}

Now, we prove the existence of the solution to the system of equations \eqref{eq:Gstareq}.
\begin{lem}[Existence of the solution to \eqref{eq:Gstareq}]
\label{lem:existsol}
There exists $ N_{\e}>0$ such that for any $n>N_{\e}$ and $h\in \mathfrak{L}^{\e}_{n}$ and $x\in \mathfrak{R}_{n,\lambda^*}(h)$ the system of equations \eqref{eq:Gstareq} admits a solution \been{
(\tilde A_n(h), \tilde \Lambda_n(h),\tilde \Lambda^{x}_n(h))
}
satisfying
\been{
\label{eq:inclusion}
\tilde A_n(h)\in [A^{-}_n(h),A^{+}_n(h)],\quad \tilde \Lambda_n(h)\in [\Lambda^{-}_n(h),\Lambda^{+}_n(h)],\qquad \tilde\Lambda^{x}_n(h)\in [\Lambda^{-,x}_n(h),\Lambda^{+,x}_n(h)].
}
with
\begin{align}
A^{-}_n(h) &:= A_n(h, C^{-}_n), & 
A^{+}_n(h) &:= A_n(h, C^{+}_n),\\
\Lambda^{-}_n(h) &:= \Lambda_n(h, A^{-}_n(h)), &
\Lambda^{+}_n(h) &:= \Lambda_n(h, A^{+}_n(h)),\\
\Lambda^{-,x}_n(h) &:= \Lambda_n(h, A^{-}_n(h)+x), &
\Lambda^{+,x}_n(h) &:= \Lambda_n(h, A^{+}_n(h)+x).
\end{align}
\end{lem}
\begin{proof}
By \eqref{eq:defApm}, $[A^{-}_n(h),A^{+}_n(h)]\subseteq (0,\Sigma_n(h))$. By Lemma \ref{lem:exists}, item (4), for any $a\in (0,\Sigma_n(h))$, the function $a\mapsto \Lambda_n(h,a)$ is well-defined, continuous, and takes value in $\R_{>0}$. So, by Lemma \ref{lem:exists}, item $(2)$, $M''_n(h,\lambda)>0$ for any $\lambda\in \R$ and, by Lemma \ref{lem:exists}, the function $M''_n(h,\cdot)$ is continuous. Hence, for fixed $h\in \mathfrak{L}^{\e}_n$, the function
\been{
F(h,a):=M^*_n(h,a)+\frac{1}{2}\log(2\pi M''_n(h,\Lambda_n(h,a))),\qquad \forall a\in[A^{-}_n(h),A^{+}_n(h)]
}
is well defined and continuous.

The upper bound in Lemma \ref{lem:bounds2} gives
\been{
\label{eq:Fbounds}
F(h,A^{-}_n(h))\leq M^*_n(h,A^{-}_n(h))+\frac{3}{4}\log(n)=C^-_n+\frac{3}{4}\log(n)=C.
}
By the lower bound in Lemma \ref{lem:bounds2}, Claims $(4)$ and $(5)$ of Lemma \ref{lem:exists}, and the lower bound in Lemma \ref{lem:insideC}
\been{
\label{eq:Fbounds0}
\aled{
 F(h,A^{+}_n(h))&\geq M^*_n(h,A^{+}_n(h))+\log\(\frac{2}{n^3\e^2}\(\Gamma_n(h)+M_n(h,\Lambda^{+}_n(h))-\Lambda^{+}_n(h) M'_n(h,\Lambda^{+}_n(h))\)\,\)\\
 &\overset{\ref{lem:exists}}{=} M^*_n(h,A^{+}_n(h))+\log\(\frac{2}{n^3\e^2}\(\Gamma_n(h)-M^*_n(h,A^{+}_n(h))\)\)\\
 &=C^+_n+\log\(\frac{2}{n^3\e^2}\(\Gamma_n(h)-C^+_n\)\,\)\geq C^+_n-2\log(n)=C.
}
}
 Since the function $a\mapsto F(h,a)$ is continuous then
\been{
F(h,[A^{-}_n(h),A^{+}_n(h)]):=\{F(h,a):\, a\in [A^{-}_n(h),A^{+}_n(h)]\}\supset \[F(h,A^{-}_n(h)),F(h,A^{+}_n(h))\].
}
Let
\been{
\mathcal{S}(h):=\{A\in [A^{-}_n(h),A^{+}_n(h)]:\, F(h,A)=C\}
}
By the upper bound in~\eqref{eq:Fbounds} and the lower bound~\eqref{eq:Fbounds0} $C\in \[F(h,A^{-}_n(h)),F(h,A^{+}_n(h))\]$. Thus $\mathcal{S}(h)\neq \emptyset$ and, since $F(h,\cdot)$ is continuous, $\mathcal{S}(h)$ is compact in $\R$. Consequently, we can take $\tilde A_n(h)=\min \mathcal{S}(h)\in [A^{-}_n(h),A^{+}_n(h)]$. Moreover, since $A_n^-$ and $A_n^+$ are measurable and $F$ is jointly continuous in the relevant arguments, the correspondence $h\mapsto \mathcal{S}(h)$ is Borel; hence $h\mapsto \tilde A_n(h)$ is Borel measurable. Finally, we take
\been{
\tilde\Lambda_n(h)=\Lambda_n(h,\tilde A_n(h)).
}
By Lemma \ref{lem:nonempty}, there exists $N_{\e}$ such that $\mathfrak{R}_{n,\lambda^*}(h)\neq \emptyset$ for $n>N_{\e}$. If $n>N_{\e}$ and $x\in \mathfrak{R}_{n,\lambda^*}(h)$, then $0<x+A^{-}_n(h)\leq x+ \tilde A_n(h)\leq x+A^{+}_n(h)\leq n\varsigma^*\wedge \Sigma_n(h) $. So we can take
\been{
\tilde\Lambda^{x}_n(h)=\Lambda_n(h,\tilde A_n(h)+x).
}
The upper and lower bounds of $\tilde\Lambda_n(h)$ and $\tilde\Lambda^{x}_n(h)$ follow from the fact that the function $a\mapsto \Lambda_n(h,a)$ is increasing (of Lemma \ref{lem:exists}, item (4)) and $\tilde A_n(h)\in [A^{-}_n(h),A^{+}_n(h)]$.
\end{proof}

\subsection{Proof of Proposition \ref{prop:finiteREM}}
We take $N_{\e}$ and $\mathfrak{R}_{n,\lambda^*}(h)$ as defined in Lemma \ref{lem:nonempty}.

If $x\in \mathfrak{R}_{n,\lambda^*}(h)$, then, by \eqref{eq:inclusion} and the definition \eqref{eq:def_R}
\been{
\tilde{A}_n(h)+x\in (0,n\varsigma^*\wedge\Sigma_n(h)).
}
By Lemma \ref{lem:exists}, item $(5)$, the function $M^*_n(h,\cdot)$ is convex in $(0,n\varsigma^*\wedge\Sigma_n(h))$, with $\dot M^*_n(h,a)=\Lambda_n(h,a)$. Thus, for any $x\in \mathfrak{R}_{n,\lambda^*}(h)$
\been{
\aled{
&M^*_n(h,\tilde{A}_n(h)+x)\geq M^*_n(h,\tilde{A}_n(h))+\dot M^*_n(h,\tilde{A}_n(h))x\geq M^*_n(h,\tilde{A}_n(h))+\tilde \Lambda_n(h)x,
}
}
and
\been{
\aled{
&M^*_n(h,\tilde{A}_n(h)+x)\leq M^*_n(h,\tilde{A}_n(h))+\dot M^*_n(h,\tilde{A}_n(h)+x)x\leq M^*_n(h,\tilde{A}_n(h))+\tilde{\Lambda}^x_n(h)x.
}
}
Since $\tilde{A}_n(h)$ solves the equation \eqref{eq:Gstareq}, we have
\been{
\aled{
M^*_n(h,\tilde{A}_n(h))+\frac{1}{2}\log(2\pi M''_n(h,\tilde\Lambda^{x}_n(h)))&=C+\frac{1}{2}\log\(\frac{M''_n(h,\tilde\Lambda^{x}_n(h))}{M''_n(h,\tilde\Lambda_n(h))}\).
}
}
Hence, combining the above bounds with the equation \eqref{eq:Gstareq}, we get
\been{
\label{eq:upper1}
\sqrt{\frac{M''_n(h,\tilde\Lambda^{x}_n(h))}{M''_n(h,\tilde\Lambda_n(h))}}\tilde{\Lambda}^x_n(h)J_n(h,\tilde{A}_n(h)+x)\leq e^{-C-\tilde \Lambda_n(h)x},\quad \forall h\in \mathfrak{L}^{\e}_{n}
}
and
\been{
\sqrt{\frac{M''_n(h,\tilde\Lambda^{x}_n(h))}{M''_n(h,\tilde\Lambda_n(h))}}\tilde{\Lambda}^x_nJ_n(h,\tilde{A}_n(h)+x)\geq e^{-C-\tilde{\Lambda}^x_n(h)x},\quad \forall h\in \mathfrak{L}^{\e}_{n}.
}
Thus, since $\tilde{A}_n(h)+x\in (0,n\varsigma^*\wedge\Sigma_n(h))$, Lemma \ref{lem:Bovier} completes the proof.

\section{Proof of Theorem \ref{them:0}}
\label{sec:5}
\noindent
In this section, we evaluate the limits $n\to \infty$ of the bounds \eqref{eq:LBP} and \eqref{eq:UBP}, proving Theorem \ref{them:0}.

Solving the system of equations \eqref{eq:Gstareq} explicitly for finite $n\in\mathbb{N}$, $x\in \mathfrak{R}_{n,\lambda^*}$, and $h\in \mathfrak{L}^{\e}_{n}$ is challenging, due to the dependence on the random sequence $h$. A key point in the analysis in this section is that the limit $n\to\infty$ of the solutions can be evaluated directly, without solving the finite-$n$ system. The finite-$n$ equations and the well-posedness of their solutions are required only to establish Proposition~\ref{prop:finiteREM}. This direct evaluation of the limit, combined with Proposition~\ref{prop:finiteREM}, proves Theorem \ref{them:0}.

We recall the definitions of the following quantities
\been{
\varsigma:=-m\psi_1+\psi_3=\res \E[\Sigma_{n}(h)],\qquad \gamma:=-\E[\log(1+\sign(h_1)m)]+\log(2)=\res \E[\Gamma_{n}(h)],
}
and the functions
\been{
\label{eq:defGG2}
G(\lambda)=\res \E[M_{n}(h,\lambda)]=\E[g(\lambda h_1)],\qquad G^*(a)=\sup_{\lambda\in \R}\(\lambda a-G(\lambda)\).
}
As usual, $G'$ is the derivative of $G$.

As in the previous section, we first provide several intermediate lemmas, organized in several subsections. We start by proving the key technical result: the roots of a sequence of invertible random functions converge almost surely to the solutions of new asymptotic equations that no longer depend on $h$. Hence, we compute the $ n\to\infty$ limit of the $\MGF$ $M_n(h,\cdot)$ and its $\FLT$ $M_n^{*}(h,\cdot)$, 
using the Strong Law of Large Numbers (SLLN) and apply the aforementioned convergence result to our model. Finally, we prove the theorem.
\subsection{The limit of the root of an invertible random field}
In this subsection, we adopt the definition of \emph{random field} from Adler and Taylor \cite[Definition 1.1.11]{Adler}. 

Let $\mathcal{M}(\R_{> 0},\overline{\R})$ denote the set of extended real-valued measurable functions with domain $\mathbb{R}_{>0}$. A random field $F$ is a $h$-measurable mapping 
\been{
F: \mathbb{R}^{\mathbb{N}} \to \mathcal{M}(\R_{> 0},\overline{\R}),
}
such that for each fixed $h\in \R^{\N}$, the function $F(h,\cdot)$ is measurable in $\R_{> 0}$, and for each fixed $x \in \mathbb{R}_{>0}$, the mapping $h \mapsto F(h,x)$ is a $h$-measurable random variable. 

Note that, with this notation, $M_n$, $M_n^*$, and their derivatives are all random fields (if we properly extend $M_n^*(h,\cdot)$ on $(\Sigma_n(h),\infty)$).
\begin{lem}
\label{lem:sollimits} 
Let $(F_n)_{n\in\mathbb{N}}$ be a sequence of extended real-valued random fields
\been{
F_n : \R^{\N} \to \mathcal{M}(\R_{> 0},\overline{\R}),}
and let 
\been{
f\in \mathcal{M}(\R_{> 0},\overline{\R})
}
be deterministic.

For each $n$ and $h \in \mathbb{R}^{\mathbb{N}}$, define
\been{
\mathfrak{X}_n(h) := \{x\in\R_{>0}:\, F_n(h,x)<\infty\}, 
\qquad 
\mathfrak{F}_n(h) :=\{F_n(h,x):\,x\in \mathfrak{X}_n(h)\},
}
and similarly
\been{
\mathfrak{X} := \{x\in\R_{>0}:\, f(x)<\infty\}, 
\qquad 
\mathfrak{F} := \{f(x):\,x\in \mathfrak{X}\}.
}
Assume:
\begin{itemize}
 \item for each $n$, $F_n(h,\cdot)$ is strictly increasing and continuous on $\mathfrak{X}_n(h)$ for $\Gae$ $h$;
 \item there exists a sequence $(m_n)_{n\in \N}$ of strictly positive numbers such that, for any $x\in \mathfrak{X}$,
 \been{
 \label{eq:Pae}
 m_n x\in \mathfrak{X}_n(h)\quad\geas,\qquad\textup{and}\qquad \lim_{n\to \infty}\res F_n(h,m_n x) =f(x),\quad \gas;
 }
 \item the function $f$ is strictly increasing and continuous on $\mathfrak{X}$.
\end{itemize}
Let $(\Phi_n)_{n\in\mathbb{N}}$ be a sequence of $h-$measurable random variables such that, for $n$ large enough,
\been{
\Phi_n(h) \in \mathfrak{F}_n(h),\quad \geas
}
and
\been{
\lim_{n\to \infty}\res \Phi_n(h) = \phi \in \mathfrak{F},\quad \gas.
}
Then:
\begin{enumerate}
 \item $\Geas$, there exists a unique $X_n(h, \Phi_n(h))\in \mathfrak{X}_n(h)$ satisfying
 \been{
 \label{eq:solfn}
 F_n(h, X_n(h, \Phi_n(h)))=\Phi_n(h);
 }
 \item there exists a unique $\hat x(\phi)\in \mathfrak{X}$ such that
 \been{
 \label{eq:solf}
 f(\hat x(\phi))=\phi;
 }
 \item 
 \been{
 \lim_{n\to \infty}\tfrac{1}{m_n} X_n(h, \Phi_n(h))=\hat x(\phi), 
 \qquad\gas.
 }
\end{enumerate}
\end{lem}
\begin{proof}
Since $F_n(h,\cdot):\mathfrak{X}_n(h)\to \mathfrak{F}_n(h)$ and $f:\mathfrak{X}\to \mathfrak{F}$ are continuous and ($\gas$) strictly increasing, they are invertible. Therefore, for any $\Phi_n(h) \in \mathfrak{F}_n(h)$ and $\phi \in \mathfrak{F}$, the equations \eqref{eq:solfn} and \eqref{eq:solf} have unique solutions in $\mathfrak{X}_n(h)$ and $\mathfrak{X}$, respectively, that are
\been{
X_n(h,\Phi_n(h))=F_n(h,\cdot)^{-1}(\Phi_n(h)),\quad \hat x(\phi)=f^{-1}(\phi).
}
 Let
\been{
\phi_{+\e}:=\phi+3\e,\qquad \phi_{-\e}:=\phi-3\e.
}
Since $f$ is continuous and strictly increasing on $\mathfrak{X}$, the set $\mathfrak{F}$ is open. Thus, we can choose $\e>0$ so that $(\phi_{-\e},\phi_{+\e})\subseteq \mathfrak{F}$, ensuring that the solutions $\hat x(\phi_{-\e})$ and $\hat x(\phi_{+\e})$ are well defined. Since $f$ is strictly increasing and continuous, $\phi\mapsto \hat x(\phi)$ is increasing. Consequently $\hat x(\phi_{-\e})\leq \hat x(\phi_{+\e})$. By \eqref{eq:Pae}, if $n$ is large enough $m_n\hat x(\phi_{-\e})\in \mathfrak{X}_n(h)$ and $m_n\hat x(\phi_{+\e})\in \mathfrak{X}_n(h)$. For such $n$, define
\been{
\aled{
&\mathfrak{H}_{n,\e}\\
&:=\left\{h\in \R^{\N}:\, \max\left\{\left|\res F_n(h,m_n\hat x(\phi_{-\e}))-f(\hat x(\phi_{-\e}))\right|,\left|\res F_n(h,m_n\hat x(\phi_{+\e}))-f(\hat x(\phi_{+\e}))\right|,|\res \Phi_n(h)-\phi|\right\}<\e\right\}.
}
}
We have
\been{
\res F_n(h,m_n\hat x(\phi_{+\e}))-\res \Phi_n(h)\geq f(\hat x(\phi_{+\e}))-\e-\res \Phi_n(h)\geq f(\hat x(\phi_{+\e}))-\phi_{+\e}=0,\qquad \forall h\in \mathfrak{H}_{n,\e},
}
and
\been{
\res F_n(h,m_n\hat x(\phi_{-\e}))-\res \Phi_n(h)\leq f(\hat x(\phi_{-\e}))+\e-\res \Phi_n(h)\leq f(\hat x(\phi_{-\e}))-\phi_{-\e}=0,\qquad \forall h\in \mathfrak{H}_{n,\e}.
}
Hence, since the function $x\mapsto F_n(h,x)$ is increasing, it must be
\been{
\label{eq:sndwitch1}
m_n\hat x(\phi_{-\e})\leq X_n(h,\Phi_n(h))\leq m_n\hat x(\phi_{+\e}),\quad \forall h\in \mathfrak{H}_{n,\e}.
}
Moreover, since $f$ is continuous and strictly increasing, $\phi\mapsto \hat x(\phi)$ is continuous. As a consequence, since $\phi_{-\e}\leq \phi\leq \phi_{+\e}$, for any $\delta>0$, there exists $\e_{\delta}>0$ such that
\been{
\label{eq:sndwitch2}
|\hat x(\phi_{-\e_{\delta}})- \hat x(\phi)|+|\hat x(\phi_{+\e_{\delta}})-\hat x(\phi)|\leq \delta.
}
Thus \eqref{eq:sndwitch1} and \eqref{eq:sndwitch2} yields
\been{
\label{eq:sndwitch3}
\{h\in \R^{\N};\,|\tfrac{1}{m_n}X_n(h,\Phi_n(h))-\hat x(\phi)|\geq \delta\}\subseteq \mathfrak{H}^c_{n,\e_{\delta}},
}
and, since $F_n(h,x)$ and $\Phi_n(h)$ converge $\Gas$ to $f(x)$ and $\phi$ respectively,
\been{
\Ph\(\limsup_{n\to \infty}\mathfrak{H}^c_{n,\e_{\delta}}\)=0.}
As a result, for any $\delta>0$,
$$
\lim_{n\to \infty}|\tfrac{1}{m_n}X_n(h,\Phi_n(h))-\hat x(\phi)|\leq \delta,\quad \gas.
$$
The above result holds for any $\delta>0$. Then, taking $\delta\to 0$, we conclude that $\tfrac{1}{m_n}X_n(h,\Phi_n(h))$ converges to $\hat x(\phi)$ $\Ph-a.s$.
\end{proof}
\subsection{The limit of $\Lambda_n(h,\cdot)$ and $A_n(h,\cdot)$}
Here we apply the Lemma \ref{lem:sollimits} to the model considered in this manuscript. We first study the analytical properties of the function $G$ and $G^*$. Then, we study the convergence of the random fields $M_n$, $M^*_n$ and other relevant random quantities. Thus, we study the convergence of the random fields $\Lambda_n$ and $A_n$, defined in \eqref{eq:Mstat} and \eqref{eq:Mstarstat}.

The derivatives of $G$ are given by
\been{
\label{eq:dersG}
G'(\lambda)=\res \E[M'_n(h,\lambda)]=\E[h_1 g'(\lambda h_1)],\quad G''(\lambda)=\res \E[M''_n(h,\lambda)]=\E[h^2_1 g''(\lambda h_1)]
}
We first state the following convergence result.
\begin{lem}
\label{lem:limits}
Under Assumption~\ref{assum:1}, the following limits hold $\Ph$-almost surely:
\begin{enumerate}
\item $\res \Sigma_{n}(h) \xrightarrow{n\to\infty} \varsigma$;
\item $\res \Gamma_{n}(h) \xrightarrow{n\to\infty} \gamma$;
\item $\res M_{n}(h,\lambda) \xrightarrow{n\to\infty} G(\lambda)$;
\item $\res M'_n(h,\lambda) \xrightarrow{n\to\infty} G'(\lambda)$;
\item $\res M''_n(h,\lambda) \xrightarrow{n\to\infty} G''(\lambda)$.
\end{enumerate}
Moreover, for $\lambda\geq 0$, we have
\been{
0 \le G(\lambda) \le 2 |\lambda|\psi_3, \quad 0 \le G'(\lambda) \le 2 \psi_3.
}
\end{lem}

\begin{proof}
The quantities $\Sigma_{n}(h)$ and $\Gamma_{n}(h)$ are sums of $n$ i.i.d. random variables which, under Assumption~\ref{assum:1}, are integrable. Therefore, the Strong Law of Large Numbers yields $(1)$ and $(2)$.

Similarly, $M_{n}(h,\lambda)$, $M'_{n}(h,\lambda)$, and $M''_{n}(h,\lambda)$ are sums of $n$ i.i.d. random variables. By the bounds established in Lemma~\ref{lem:bounds}, their summands are integrable, and the Strong Law of Large Numbers therefore gives the corresponding convergences as well.

Finally, the bounds on $g$ and $g'$ in Lemma~\ref{lem:bounds} imply the corresponding inequalities for $G(\lambda)$ and $G'(\lambda)$.
\end{proof}
The functions $G$ and $G^*$ inherit the structural properties established for $M_n(h,\cdot)$ and $M^*_n(h,\cdot)$ stated in Lemma \ref{lem:exists}. As before, we denote the derivative of $G^*$ by $\dot G^*$.
\begin{lem}
\label{lem:exists2}
The function $G$ is continuous, twice differentiable, and verifies the following
\begin{enumerate}
\item $G(0)=G'(0)=0$;
\item $G''(\lambda)>0$ for any $\lambda\in \R$;
\item $\{G'(\lambda):\, \lambda \in \R_{> 0}\}= (0,\varsigma)$;
\item there exists a continuous and strictly increasing function $\hat \lambda:[0,\varsigma)\to \R_{\geq 0}$ such that
\been{
\label{eq:statG}
 G'(\hat \lambda(a))= a,\qquad \forall a\in [0,\varsigma).
}
Moreover $\hat\lambda(0)=0$ and $\hat\lambda((0,\varsigma))=\R_{> 0}$.
\end{enumerate}
The $\FLT$ $G^*$ satisfies
\begin{enumerate}
\setcounter{enumi}{4}
\item for any $a\in (0,\varsigma)$
\been{
\label{eq:Gstarstat}
G^*(a)=a \hat \lambda(a)- G(\hat \lambda(a)),\qquad \dot G^*(a)=\hat \lambda(a);
}
\item $G^*$ is strictly increasing in $[0,\varsigma)$;
\item $\{G^*(a):\, a\in (0,\varsigma)\}=(0,\gamma)$
\item there exists a continuous increasing function $\hat a:(0,\gamma)\to (0,\varsigma)$ such that
\been{
\label{eq:ga2}
G^*(\hat a(c))=c,\quad \forall c\in (0,\gamma).
}
\end{enumerate}
\end{lem}
\begin{proof}
The continuity and the differentiability follow from the well-posedness and finiteness of the expectation values in \eqref{eq:dersG}. We now proceed to prove the remaining claims separately.
\begin{proof}[Proof of Claim (1)]The equality \eqref{eq:g2trivial}, the definition \eqref{eq:defGG2}, and \eqref{eq:dersG} yield the claim.
\end{proof}
\begin{proof}[Proof of Claim (2)]
By Lemma \ref{lem:bounds}, $h_1^2g''(\lambda h_1)>0$ for any $h_1\neq 0$. Hence \eqref{eq:dersG} and Assumption \ref{assum:1} yield the claim.
\end{proof}
\begin{proof}[Proof of Claim (3)]
Since $G''(\lambda)>0$, then $G'$ is strictly increasing. So $\inf_{\lambda \in \mathbb{R}_{\ge 0}} G'(\lambda)= G'(0) = 0$. Moreover, since $M'_n(h,\cdot)$ is also increasing, the Monotone Convergence Theorem together with \eqref{eq:dersG} yields
\been{
\sup_{\lambda \in \mathbb{R}} G'(\lambda) = \lim_{\lambda \to \infty} G'(\lambda) 
= \lim_{\lambda \to \infty} \res \E[M'_n(h,\lambda)] 
= \res \E\Big[\lim_{\lambda \to \infty} M'_n(h,\lambda)\Big] 
= \res \E[\Sigma_n(h)] = \varsigma.
}
\end{proof}
\begin{proof}[Proof of Claim (4)]
The Claim $(1)$, $(2)$ and $(3)$ of this Lemma imply that the restriction $G':[0,\infty)\to [0,\varsigma)$ is invertible. So, we define
$\hat{\lambda} := (G')^{-1} : [0,\varsigma) \to [0,\infty)$, which is continuous and strictly increasing, since $G'$ is continuous and strictly increasing. By definition, $\hat{\lambda}(a)$ is the unique solution in $\mathbb{R}_{>0}$ of \eqref{eq:statG}. Moreover, since $G'(0)=0$, $\hat{\lambda}(0)=0$, and, since $\hat{\lambda}$ is strictly increasing, $\hat{\lambda}(a)>0$ for $a>0$.
\end{proof}
\begin{proof}[Proof of Claim (5)]
Since the function $\R:\lambda\mapsto \lambda a-G(\lambda)$ is strictly concave and differentiable, the stationary point is also the supremum. Moreover, since $G''(\lambda)>0$ for any $\lambda\in \R$, by the Implicit Function Theorem the function $a\mapsto \hat \lambda(a)$ is differentiable.

 Consequently
\been{
\dot G^*(a)=\dot{\hat \lambda}(a)\(a-G'(\hat \lambda(a))\)+\hat \lambda(a)=\hat \lambda(a)>0,\quad \forall a\in (0,\varsigma).
}
\end{proof}
\begin{proof}[Proof of Claim (6)]
Claim $(4)$ and Claim $(5)$ of this Lemma prove this Claim.
\end{proof}
\begin{proof}[Proof of Claim (7)]
Since $G^*$ is increasing in $(0,\varsigma)$, we have
\been{
\inf_{a\in [0,\varsigma)}G^*(a)=G^*(0)=-G(\hat \lambda(0))=0
}
and
\been{
\aled{
\sup_{a\in [0,\varsigma)}G^*(a)&=\lim_{a\to \varsigma}G^*(a)=\sup_{\lambda\in \R}(\varsigma \lambda-G(\lambda))\\&=\sup_{\lambda\in \R}\(\lambda \E[|h_1|]-\E[\log((1+m) e^{\lambda h_1}+(1-m)e^{-\lambda h_1})]+\log(2)\)\\
&=\sup_{\lambda\in \R}\(-\E[\log((1+\sign(h_1)m)+(1-\sign(h_1)m)e^{-2\lambda|h_1|})]\)+\log(2).
}
}
Hence, the supremum in $\lambda$ is achieved by taking the limit $\lambda\to \infty$
\been{
\aled{
&\sup_{\lambda\in \R}\(-\E[\log((1+\sign(h_1)m)+(1-\sign(h_1)m)e^{-2\lambda|h_1|})]+\log(2)\)\\
&=-\E[\log(1+\sign(h_1)m)]+\log(2)=\gamma.
}
}
\end{proof}
\begin{proof}[Proof of Claim (8)]
By Claim $(4)$ and Claim $(5)$ of this lemma, the function $a\mapsto G^*(a)$ has a strictly positive derivative for $a\in (0,\varsigma)$. Hence it is invertible from $(0,\varsigma)$ to $G^*((0,\varsigma))=(0,\gamma)$, where the last equality is proved in Claim $(7)$ of this lemma.
\end{proof}
\end{proof}
We now show that, in the limit $n \to \infty$, the solution of the system of equations \eqref{eq:Gstareq} is determined from the functions $G^*$ and $G'$. In the following, given $a\in \R$, we denote by $\Lambda_n(h,a)$ the solution to \eqref{eq:Mstat}, for a fixed $h\in \R^{\N}$ and $n\in \N$, and by $\hat \lambda(a)$ the solution to \eqref{eq:statG}.

We use the above lemma to prove the following two convergence results.
\begin{lem}
\label{lem:convLnA}
Given $a\in (0,\varsigma)$ and a $h-$measurable sequence $(A_n)_{n\in \N}$ such that
\been{
\label{eq:conva}
\lim_{n\to \infty}\res A_n(h)=a,\qquad \gas
}
we have
\been{
\label{eq:conveta}
\lim_{n\to \infty}\Lambda_n(h,A_n(h))=\hat \lambda(a),\qquad \gas
}
and 
\been{
\label{eq:convGstar}
\lim_{n\to \infty} \res M^*_n(h,A_n(h))=G^*(a),\qquad \gas.
}
\end{lem}
\begin{proof}
If $a\in (0,\varsigma)$, then, since $\res \Sigma_n(h)\xrightarrow{n\to \infty} \varsigma$ $\Gas$,
\been{
A_n(h)\in (0,\Sigma_n(h)),\quad \geas.
}
To apply Lemma~\ref{lem:sollimits}, set
\been{
\mathfrak{X}_n=\mathfrak{X}=\R_{>0},\quad \mathfrak{F}_n=(0,\Sigma_n(h)),\quad \mathfrak{F}=(0,\varsigma),\quad F_n=M'_n,\quad f=G',\quad m_n= 1,\quad \Phi_n=A_n,\quad \phi=a,
}
and
\been{
X_n(h,\Phi_n(h))=\Lambda_n(h,A_n(h)),\quad \hat{x}(\phi)=\hat \lambda(a).
}
Then, by Lemmas~\ref{lem:exists}, \ref{lem:limits}, and \ref{lem:exists2}, all assumptions of Lemma~\ref{lem:sollimits} are verified, proving \eqref{eq:conveta}.

For \eqref{eq:convGstar}, the equations \eqref{eq:Mstarstat} and \eqref{eq:Gstarstat} and the triangular inequality give
\been{
\aled{
\left|\res M^*_n(h,A_n(h))-G^*(a)\right|&=\left| \res \(A_n(h)\Lambda_n(h,A_n(h))-M_{n}(h,\Lambda_n(h,A_n(h)))\)-\(a\hat \lambda(a)-G(\hat \lambda(a))\)\right|\\
&\leq \res A_n(h)|\Lambda_n(h,A_n(h))-\hat \lambda(a)|+\left|\res A_n(h)-a\right|\hat \lambda(a)\\
&+\res\left| M_{n}(h,\Lambda_n(h,A_n(h)))- M_{n}(h,\hat \lambda(a)) \right|+\left|\res M_{n}(h,\hat \lambda(a))- G(\hat \lambda(a))\right|
}
}
By the Claim $(3)$ of Lemma \ref{lem:exists}
\been{
\aled{
\res|M_{n}(h,\Lambda_n(h,A_n(h)))- M_{n}(h,\hat \lambda(a))|&\leq |\Lambda_n(h,A_n(h))-\hat\lambda(a)|\res \sup_{\lambda\geq 0}|M'_n(h,\lambda)|\\
&\leq \res \Sigma_n(h)|\Lambda_n(h,A_n(h))-\hat\lambda(a)|.
}
}
So the convergences \eqref{eq:conva}, \eqref{eq:conveta}, and Claim $(1)$ and $(3)$ of Lemma \ref{lem:limits} prove the limit \eqref{eq:convGstar}.
\end{proof}
In the following, given $c\in \R$, we denote by $A_n(h,c)$ the solution to \eqref{eq:GA}, for a fixed $h\in \R^{\N}$ and $n\in \N$, and by $\hat a(c)$ the solution to \eqref{eq:ga2}.
\begin{lem}
\label{lem:convAnC}
Given $c\in (0,\gamma)$ and a $h-$measurable sequence $(C_n)_{n\in \N}$ such that
\been{
\label{eq:conva2}
\lim_{n\to \infty}\res C_n(h)=c,\quad \gas
}
we have
\been{
\label{eq:conveta2}
\lim_{n\to \infty}\res A_n(h,C_n(h))=\hat{a}(c),\qquad \gas
}
\end{lem}
\begin{proof}
If $c\in (0,\gamma)$, then, since $\res \Gamma_n(h)\xrightarrow{n\to \infty} \gamma$, $\Gas$, by \eqref{eq:conva2}
\been{
C_n(h)\in (0,\Gamma_n(h)),\quad \geas.
} 
Let us define
\been{
\bar M^*_n(h,a):=
\begin{cases}
     M^*_n(h,a),\quad &\textup{if }a\in [0,\Sigma_n(h)];\\
    \infty,\quad &\textup{if }a>\Sigma_n(h).
\end{cases}
}
To apply Lemma~\ref{lem:sollimits}, set
\been{
\mathfrak{X}_n=(0,\Sigma_n(h)),\quad \mathfrak{X}=(0,\varsigma),\quad \mathfrak{F}_n=(0,\Gamma_n(h)),\quad \mathfrak{F}=(0,\gamma),\quad F_n=\bar M^*_n,\quad f=G^*,\quad m_n=n,
}
and
\been{
\Phi_n=C_n,\quad \phi=c,\quad X_n(h,\Phi_n(h))=A_n(h,C_n(h)),\quad \hat{x}(\phi)=\hat{a}(c).
}
Then, by Lemmas~\ref{lem:exists}, equation \eqref{eq:convGstar} in Lemma \ref{lem:convLnA}, and \ref{lem:exists2}, all assumptions of Lemma~\ref{lem:sollimits} are verified, proving \eqref{eq:conveta2}.
\end{proof}
\subsection{Proof of Theorem \ref{them:0}}
Lemma \ref{lem:exists2} proves the existence of the solution of \eqref{eq:solac0}. If $c\in (0,\gamma)$ and $m\in (-1,1)$, there exists $\e>0$ such that:
\been{
c\in (2\e\gamma,(1-2\e)\gamma),\qquad m\in [-1+2\e,1-2\e].
}
Consider a sequence of random variables $(C_n)_{n\in \N}$ verifying \eqref{eq:CcondREM0}. By \eqref{eq:CcondREM0}, Lemma \ref{lem:limits}, and the choice of $\e$, we have
\been{
C_n(h)\in (\e\Gamma_n(h),(1-\e)\Gamma_n(h)),\quad \geas.
}
From $C_n$, define $C^+_n$ and $C^-_n$ as in \eqref{eq:defCpm}, and define $A^{-}_n$ and $A^{+}_n$ as in \eqref{eq:defApm}.

Given $\lambda^*>1$, let $N_{\e}$, $\mathfrak{L}^{\e}_{n}$, and $\mathfrak{R}_{n,\lambda^*}(h)$ be the objects defined in Proposition \ref{prop:finiteREM}. Let
\been{
\bar{\mathfrak{L}}^{\e}_N:=\bigcap^{\infty}_{n=N}\mathfrak{L}^{\e}_{n},\qquad \bar{\mathfrak{R}}_{N,\lambda^*}(h):=\bigcap^{\infty}_{n=N}\mathfrak{R}_{n,\lambda^*}(h).
}
If $h\in  \bar{\mathfrak{L}}^{\e}_N$ and $x\in \bar{\mathfrak{R}}_{N,\lambda^*}(h)$ then the setting of Proposition \ref{prop:finiteREM} is eventually verified. Thus, in this set, we can evaluate the $n\to \infty$ limit of the bounds \eqref{eq:LBP} and \eqref{eq:UBP}. Hence, we take the limit $N\to \infty$ and show that $\bar{\mathfrak{L}}^{\e}_N$ converges to a set of probability $1$ and $ \bar{\mathfrak{R}}_{N,\lambda^*}(h)$ converges to $\R$.

In the following, given $N>N_{\e}$, $h\in \bar{\mathfrak{L}}^{\e}_N$, $x\in \bar{\mathfrak{R}}_{N,\lambda^*}(h)$, and $n>N$, we denote by $(\tilde A_n(h),\tilde \Lambda_n(h),\tilde \Lambda^{x}_n(h))$ a solution of \eqref{eq:Gstareq}. 

We split the proof in several lemmas. We first compute the limit $n\to \infty$ of the solution $(\tilde A_n(h),\tilde \Lambda_n(h),\tilde \Lambda^{x}_n(h))$
\begin{lem}
We have
\been{
\label{eq:convAnCCC0}
\lim_{n\to \infty}\res A^{-}_n(h)=\lim_{n\to \infty}\res A^{+}_n(h)=\lim_{n\to \infty}\res \( A^{-}_n(h)+x\)=\lim_{n\to \infty}\res \( A^{+}_n(h)+x\)=\hat{a}(c),\quad \gas ,
}
for any fixed $x\in \R$.

Moreover, given any $N> N_{\e}$, for $\Gae$ $h\in \bar{\mathfrak{L}}^{\e}_N$ and any fixed $x\in \bar{\mathfrak{R}}_{N,\lambda^*}(h)$, 
\been{
\label{eq:limitsfinal}
\lim_{n\to \infty}\res \tilde A_n(h)=\hat{a}(c),\qquad \lim_{n\to \infty}\tilde \Lambda_n(h)=\lim_{n\to \infty}\tilde \Lambda^{x}_n(h)=\hat{\lambda}(\hat{a}(c)).
}
\end{lem}
\begin{proof}
 By the convergence \eqref{eq:CcondREM0}
\been{
\lim_{n\to \infty}\res C^+_n(h)=\lim_{n\to \infty}\res C^-_n(h)=\lim_{n\to \infty}\res C_n(h)=c,\quad \gas .
}
Therefore, Lemma \ref{lem:convAnC} gives
\been{
\label{eq:convAnCCC}
\lim_{n\to \infty}\res A^{-}_n(h)=\lim_{n\to \infty}\res A^{+}_n(h)=\hat{a}(c),\quad \gas .
}
Consequently, for any $x\in \R$,
\been{
\label{eq:convAnCCCx}
\lim_{n\to \infty}\res \( A^{-}_n(h)+x\)=\lim_{n\to \infty}\res \( A^{+}_n(h)+x\)=\hat{a}(c),\quad \gas .
}
The above limits and Lemma \ref{lem:convLnA} yield
\been{
\label{eq:convLnCCC}
\lim_{n\to \infty} \Lambda_n(h,A^{-}_n(h))=\lim_{n\to \infty} \Lambda_n(h,A^{+}_n(h))=\lim_{n\to \infty} \Lambda_n(h,A^{-}_n(h)+x)=\lim_{n\to \infty} \Lambda_n(h,A^{+}_n(h)+x)=\hat{\lambda}(\hat{a}(c)),\quad \gas.
}
 Using the bounds \eqref{eq:inclusion} in Lemma \ref{lem:existsol}, the limit \eqref{eq:convAnCCC}, \eqref{eq:convAnCCCx}, and \eqref{eq:convLnCCC} completes the proof.
\end{proof}
\begin{lem}
Given any $N>N_{\e}$, for $\Gae$ $h\in \bar{\mathfrak{L}}^{\e}_N$ and any fixed $x\in \bar{\mathfrak{R}}_{N,\lambda^*}(h)$, 
\been{
\lim_{n\to \infty}\frac{M_n''(h,\tilde\Lambda_n(h))}{M_n''(h,\tilde\Lambda_n^x(h))}=1
}
\end{lem}
\begin{proof}
Let $(\Lambda_n(h))_{n\in\N}$ be an $h$-measurable sequence such that
\been{
\Lambda_n(h)\longrightarrow \hat\lambda(\hat a(c)), \qquad \gas.
}
By the Lipschitz bound \eqref{eq:g2LB_lip},
\been{
\label{eq:M2moving_aux}
\begin{aligned}
\left|
\frac{1}{n}M_n''(h,\Lambda_n(h))
-
\frac{1}{n}M_n''(h,\hat\lambda(\hat a(c)))
\right|
&\leq
2\left(\frac{1}{n}\sum_{i=1}^n |h_i|^3\right)
\left|\Lambda_n(h)-\hat\lambda(\hat a(c))\right|.
\end{aligned}
}
Since $\E[|h_1|^3]<\infty$ by Assumption~\ref{assum:1}, the strong law of large numbers yields
\been{
\label{eq:SLLNthird}
\lim_{n\to \infty} \res\sum_{i=1}^n |h_i|^3=\E[|h_1|^3], \qquad \gas.
}
Hence, by the above convergence and \eqref{eq:M2moving_aux},
\been{
\label{eq:M2moving_aux2}
\lim_{n\to \infty} \res\left|M_n''(h,\Lambda_n(h))
-
M_n''(h,\hat\lambda(\hat a(c)))
\right|
= 0, \qquad \gas.
}
Combining the above result with Claim $(5)$ of Lemma~\ref{lem:limits}, we obtain
\been{
\label{eq:M2moving}
\aled{
&\lim_{n\to \infty} \left|\res M_n''(h,\Lambda_n(h))
-
G''(\hat\lambda(\hat a(c)))
\right|\\
&\leq \lim_{n\to \infty} \(\res\left|M_n''(h,\Lambda_n(h))
-
M_n''(h,\hat\lambda(\hat a(c)))
\right|+ \left|\res M_n''(h,\hat\lambda(\hat a(c)))
-
G''(\hat\lambda(\hat a(c)))
\right|\)
= 0, \quad \gas.
}
}
If $h\in \bar{\mathfrak{L}}^{\e}_N$ and $x\in \bar{\mathfrak{R}}_{N,\lambda^*}(h)$, for any $N\in \N$ large enough, we can apply \eqref{eq:M2moving} by taking both $\Lambda_n(h)=\tilde\Lambda_n(h)$ and $\Lambda_n(h)=\tilde\Lambda_n^x(h)$. Using \eqref{eq:limitsfinal}, for $\Gae$ $h\in \bar{\mathfrak{L}}^{\e}_N$ and fixing $x\in \bar{\mathfrak{R}}_{N,\lambda^*}(h)$, 
\been{
\frac{1}{n}M_n''(h,\tilde\Lambda_n(h))
\xrightarrow{n\to\infty}
G''(\hat\lambda(\hat a(c))),
\qquad
\frac{1}{n}M_n''(h,\tilde\Lambda_n^x(h))
\xrightarrow{n\to\infty}
G''(\hat\lambda(\hat a(c))).
}
Since $G''(\hat\lambda(\hat a(c)))>0$, we conclude that, for $N\in \N$ large enough, for $\Gae$ $h\in \bar{\mathfrak{L}}^{\e}_N$ and any fixed $x\in \bar{\mathfrak{R}}_{N,\lambda^*}(h)$,
\been{
\label{eq:MM}
\lim_{n\to\infty}\frac{M_n''(h,\tilde\Lambda_n(h))}{M_n''(h,\tilde\Lambda_n^x(h))}
=
\lim_{n\to\infty}\frac{\frac{1}{n}M_n''(h,\tilde\Lambda_n(h))}{\frac{1}{n}M_n''(h,\tilde\Lambda_n^x(h))}
= 1.
}
\end{proof}
The above lemmas hold under the restriction $h\in \bar{\mathfrak{L}}^{\e}_N$ and $x\in \bar{\mathfrak{R}}_{N,\lambda^*}(h)$. We want to extend the result to almost all $h$ and all $x\in \R$. The next lemma shows that taking the limit $N\to \infty$ yields the desired extension.
\begin{lem}
\been{
\label{eq:r2}
 \Ph(\liminf_{N\to \infty}\bar{\mathfrak{L}}^{\e}_{N})=1
}
and
 \been{
 \label{eq:r10}
 \aled{
 \liminf_{\lambda^*\to \infty}\liminf_{N\to \infty}\bar{\mathfrak{R}}_{N,\lambda^*}(h)&=\R,\quad \gas.
 }
 }
\end{lem}
\begin{proof}
We start by proving \eqref{eq:r2}. Let us define
 \been{
 \bar{\mathfrak{L}}^{\e}_{N,1}:=\bigcap^{\infty}_{n=N}\left\{h\in \R^{\N}:\quad 2\pi\sum^n_{i=1}h^2_i\leq n^{3/2},\quad \Sigma_n(h)\in (n^{4/5}, n^{3/2}) \right\}
 }
and
 \been{
{\mathfrak{L}}^{\e}_{n,2}:=\left\{h\in \R^{\N}:\quad 2\pi\min_{1\leq i\leq n} h^2_i\geq \frac{16}{n^5\e^8} \right\},\qquad \bar{\mathfrak{L}}^{\e}_{N,2}:=\bigcap^{\infty}_{n=N}{\mathfrak{L}}^{\e}_{n,2}.
 }
 Thus $\bar{\mathfrak{L}}^{\e}_N=\bar{\mathfrak{L}}^{\e}_{N,1}\cap \bar{\mathfrak{L}}^{\e}_{N,2}$. We have
 \been{
 \label{eq:part00}
 \Ph(\liminf_{N\to \infty}\bar{\mathfrak{L}}^{\e}_{N,1})\geq \Ph\(\left\{h\in \R^{\N}:\, \lim_{n\to\infty}\sum^n_{i=1}\res h^2_i\in (0,\infty), \lim_{n\to \infty}\res\Sigma_n(h)\in (0,\infty) \right\}\)=1.
 }

 Since
\been{
({\mathfrak L}_{n,2}^{\e})^c
=
\bigcup_{i=1}^n
\left\{
h_i^2<
\frac{16}{2\pi n^5 \e^8}
\right\},
}
by the union bound, independence, and the Assumption \ref{assum:1}
\been{
\Ph(({\mathfrak L}_{n,2}^{\e})^c)
\le
n\,\Ph\!\left(
|h_1|<
\frac4{\sqrt{2\pi}\,n^{5/2}\e^4}
\right)
\le C_\e n^{-3/2}.
}
Therefore
\been{
\sum_{n=1}^\infty \Ph(({\mathfrak L}_{n,2}^{\e})^c)<\infty.
}
Hence, by Borel--Cantelli,
\been{
\label{eq:part0}
\Ph\!\left(\limsup_{n\to\infty}({\mathfrak L}_{n,2}^{\e})^c\right)=0\Longrightarrow
\Ph\!\left(\liminf_{N\to\infty}\bar{\mathfrak L}_{N,2}^{\e}\right)=1.
}
So, \eqref{eq:part00} and \eqref{eq:part0} give \eqref{eq:r2}.

Now, we prove \eqref{eq:r10}. Note that $\varsigma^*=G'(\lambda^*)$. Hence, by Claim $(2)$ and $(3)$ of Lemma \ref{lem:exists2} and Claim $(1)$ of Lemma \ref{lem:limits}
\been{
\label{eq:limsigma}
\lim_{\lambda^*\to \infty}\varsigma^*=\sup_{\lambda^*}\varsigma^*=\varsigma=\lim_{\lambda^*\to \infty}\lim_{n\to \infty}\varsigma^*\wedge \res\Sigma_n(h),\quad \gas.
}
The limits \eqref{eq:convAnCCC0} give
\been{
\liminf_{n\to \infty}A^{-}_{n}(h)=\infty,\quad \gas,
}
and, since $\lim_{n\to \infty }\res A^{+}_{n}(h)=\hat a(c)<\varsigma=\lim_{\lambda^*\to \infty}\lim_{n\to \infty}\varsigma^*\wedge \res\Sigma_n(h)$, 
\been{
\lim_{\lambda^*\to \infty}\liminf_{n\to \infty} \((n\varsigma^*\wedge \Sigma_n(h))-A^{+}_{n}(h)\)=\lim_{\lambda^*\to \infty}\liminf_{n\to \infty} \(n\(\varsigma^*\wedge \res\Sigma_n(h)-\res A^{+}_{n}(h)\)\)=\infty,\quad \gas.
}
Therefore, taking the limit $N\to \infty$ and $\lambda^*\to \infty$ in the definition \eqref{eq:def_R} of $\bar{\mathfrak{R}}_{N,\lambda^*}(h)$, we get
 \been{
 \label{eq:r1}
 \aled{
 \liminf_{\lambda^*\to \infty}\liminf_{N\to \infty}\bar{\mathfrak{R}}_{N,\lambda^*}(h)&=\lim_{\lambda^*\to \infty}\(\,-\liminf_{n\to \infty} A^{-}_{n}(h),\liminf_{n\to \infty} \((n\varsigma^*\wedge \Sigma_n(h))-A^{+}_{n}(h)\)\,\)=\R.
 }
 }
\end{proof}
We can finally prove the Theorem \ref{them:0}
\begin{proof}[Proof of Theorem \ref{them:0}]
Given a Borel set $\mathfrak{U}$, let
\been{
\mathbf{K}_n(h,\mathfrak{U}):= e^{C_n(h)}\Ps(\{\sigma:\,H_n(h,\sigma)-\tilde A_n(h)\in \mathfrak{U}\})
}
Combining Proposition \ref{prop:finiteREM} with the limits \eqref{eq:limitsfinal} and \eqref{eq:MM}, for any $N\in \N\cap(N_{\e},\infty)$, we get that for $\Gae\,h\in \bar{\mathfrak{L}}^{\e}_N$ and any fixed $x\in \bar{\mathfrak{R}}_{N,\lambda^*}(h) $
\been{
\lim_{n\to \infty} \mathbf{K}_n(h,[x,\infty))=\D_{\hat{\lambda}(\hat{a}(c))}([x,\infty)).
}
Hence, by the countable intersection of events with probability $1$, the above limit holds simultaneously for any $N\in\N\cap(N_{\e},\infty)$ and $\lambda^*\in\N\cap(1,\infty)$
\been{
\label{eq:almost_limit2}
\lim_{n\to \infty} \mathbf{K}_n(h,[x,\infty))=\D_{\hat{\lambda}(\hat{a}(c))}([x,\infty))\quad \forall x\in \bar{\mathfrak{R}}_{N,\lambda^*}(h)\cap \mathbb{Q},\quad \Gae\,h\in \bar{\mathfrak{L}}^{\e}_N.
}
So, taking $N\to \infty$ and $\lambda^*\to \infty$ , the limits \eqref{eq:r1} and \eqref{eq:r2} give
 \been{
 \label{eq:set1}
 \lim_{n\to \infty} \mathbf{K}_n(h,[x,\infty))=\D_{\hat{\lambda}(\hat{a}(c))}([x,\infty)),\quad \forall x\in \mathbb{Q},\quad \gas.
 }
Hence, for every $q<r$, $q,r\in\mathbb{Q}$,
\been{
\mathbf{K}_n(h,[q,r))
=
\mathbf{K}_n(h,[q,\infty))-\mathbf{K}_n(h,[r,\infty))
\xrightarrow{n\to \infty}
\D_{\hat{\lambda}(\hat{a}(c))}([q,r)),
\qquad \gas.
}
Fix a bounded interval $\mathfrak{K}:=[-k,k)$, for some $k\in \mathbb{Q}\cap(0,\infty)$. Then
\been{
\mathcal{I}_{\mathfrak{K}}:=\{[q,r)\cap \mathfrak{K}:\ q<r,\ q,r\in\mathbb{Q}\}\cup\{\mathfrak{K}\}
}
is a covering semiring of $\mathfrak{K}$. Hence, by the convergence-determining
class theorem, $\mathcal{I}_{\mathfrak{K}}$ is convergence-determining for weak convergence on $\mathfrak{K}$ \cite[Appendix A2.3, Proposition A2.3.IV]{DaleyVereJones2003}. Therefore $\mathbf{K}_n(h,\cdot)\big|_{\mathfrak{K}}$ converges weakly to $ \D_{\hat{\lambda}(\hat{a}(c))}\big|_{\mathfrak{K}}$.

Since every continuous and compactly supported function $f$ has support contained in some bounded interval $\mathfrak{K}$,
it follows that $\mathbf{K}_n(h,\cdot)$ converges vaguely to $\D_{\hat{\lambda}(\hat{a}(c))}$.

Comparing the definitions \eqref{eq:solac0}, we get
\been{
\hat a(c)=\tilde a,\qquad \hat{\lambda}(\hat{a}(c))=\tilde \lambda,
}
completing the proof of Theorem \ref{them:0}.
\end{proof}
\bibliographystyle{alpha}
\bibliography{REMbib}

\end{document}